\documentclass{aims}
\usepackage{amsmath}
  \usepackage{paralist}
  \usepackage{graphics} 
  \usepackage{epsfig} 
\usepackage{graphicx}  \usepackage{epstopdf}
 \usepackage[colorlinks=true]{hyperref}
\hypersetup{urlcolor=blue, citecolor=red}

\newtheorem{theorem}{Theorem}[section]

\newtheorem{lemma}[theorem]{Lemma}
\newtheorem{proposition}{Proposition}

\theoremstyle{definition}

\newtheorem{remark}{Remark}

\usepackage{subfig}
\usepackage{comment}
\usepackage{verbatim}
\usepackage[dvipsnames]{xcolor}
\usepackage{enumitem}
\usepackage[numbers]{natbib} 
\usepackage{algorithm,algpseudocode}
\usepackage{setspace}

\newtheorem{assum}{Assumption}[section]
\newtheorem{cor}{Corollary}[section]
\newcommand{\pos}{\mathrm{pos}}
\newcommand{\pri}{\mathrm{prior}}
\newcommand{\Rm}{\mathrm{m}}
\newcommand{\EE}{\mathbb{E}}
\newcommand{\RR}{\mathbb{R}}
\newcommand{\Cov}{\mathrm{Cov}}
\newcommand{\MCF}{\mathcal{F}}

\newcommand{\MCG}{\mathcal{G}}

\newcommand{\sm}{\mathsf{m}}

\newcommand{\rd}{\,\mathrm{d}}

\title[Weighted Ensemble Kalman Inversion]{Ensemble Kalman Inversion for nonlinear problems: weights, consistency, and variance bounds}

\author[Zhiyan Ding, Qin Li and Jianfeng Lu]{}


 \email{zding49@math.wisc.edu}
 \email{qinli@math.wisc.edu}
 \email{jianfeng@math.duke.edu}

\thanks{\textcolor{black}{Zhiyan Ding and Qin Li are supported in part by NSF CAREER DMS-1750488, NSF TRIPODS 1740707 and Wisconsin Data Science Initiative. The work of Jianfeng Lu is supported in part by National Science Foundation via grants DMS-1454939 and DMS-2012286. All three authors thank the two anonymous referees for the very helpful suggestions.}}

\thanks{$^*$ Corresponding author: Zhiyan Ding}

\begin{document}
\maketitle

\centerline{\scshape Zhiyan Ding}
\medskip
{\footnotesize
 \centerline{Department of Mathematics}
   \centerline{University of Wisconsin-Madison}
   \centerline{Madison, WI 53705 USA}
} 

\medskip

\centerline{\scshape Qin Li$^\dagger$}
\medskip
{\footnotesize
 \centerline{Department of Mathematics}
   \centerline{University of Wisconsin-Madison}
   \centerline{Madison, WI 53705 USA}
} 

\medskip

\centerline{\scshape Jianfeng Lu$^*$}
\medskip
{\footnotesize
 \centerline{Department of Mathematics, Department of Physics, and Department of Chemistry}
   \centerline{Duke University}
   \centerline{Durham, NC 27708 USA}
} 

\bigskip

\begin{abstract}
\textcolor{black}{Ensemble Kalman Inversion (EnKI)~\cite{Iglesias_2013} and Ensemble Square Root Filter (EnSRF)~\cite{ESRF2003} are popular sampling methods for obtaining a target posterior distribution. They can be seem as one step (the analysis step) in the data assimilation method Ensemble Kalman Filter~\cite{Evensen:2006:DAE:1206873,BR2010}. Despite their popularity, they are, however, not unbiased when the forward map is nonlinear~\cite{ding2019ensemble,doi:10.1137/140981319,doi:10.1137/140984415}. Important Sampling (IS), on the other hand, obtains the unbiased sampling at the expense of large variance of weights, leading to slow convergence of high moments.}

\textcolor{black}{We propose WEnKI and WEnSRF, the weighted versions of EnKI and EnSRF in this paper. It follows the same gradient flow as that of EnKI/EnSRF with weight corrections. Compared to the classical methods, the new methods are unbiased, and compared with IS, the method has bounded weight variance. Both properties will be proved rigorously in this paper. We further discuss the stability of the underlying Fokker-Planck equation. This partially explains why EnKI, despite being inconsistent, performs well occasionally in nonlinear settings. Numerical evidence will be demonstrated at the end. }
\end{abstract}

\section{Introduction}
How to sample from an intractable distribution is a classical challenge emerging from Bayesian statistics, machine learning, computational physics, among many other areas. Denote $\mathcal{G}: \mathcal{X}\rightarrow\mathcal{Y}$ a forward map between separable Hilbert spaces $\mathcal{X}$ and $\mathcal{Y}$. While the forward problem amounts to finding $\mathcal{G}(u)$ for every $u\in\mathcal{X}$, the inverse problem amounts to reconstructing the unknown parameters $u$ from the observation $y$. Sampling provides a probability perspective for such reconstruction procedure. Throughout the paper we set $\mathcal{X}=\mathbb{R}^L$ and $\mathcal{Y} = \mathbb{R}^K$.

Let $y$ be the collected data. It is generated from the forward map $\mathcal{G}$ acting on $u$ with added Gaussian noise $\eta$ that is assumed to be independent of $u$:
\[
y=\mathcal{G}(u)+\eta\,,\quad\text{with}\quad \eta\sim \mathcal{N}(0,\Gamma)\,.
\]
Throughout, we assume $\mathcal{G}$ is sufficiently smooth and its gradient is denoted by
\[
[\nabla\mathcal{G}(u)]_{i,j} = \partial_j\mathcal{G}_i,\quad \forall 1\leq i\leq K,\,1\leq j\leq L\,.
\]

To find $u$ using $y$, a typical approach is to perform minimization. We denote the least-squares functional $\Phi(\cdot;y):\mathcal{X}\rightarrow\mathbb{R}$ by
\begin{equation}\label{eqn:phi}
\Phi(u;y)=\frac{1}{2}\left|y-\mathcal{G}(u)\right|^2_\Gamma=\frac{1}{2}\left(y-\MCG(u)\right)^\top\Gamma^{-1}\left(y-\MCG(u)\right)\,,
\end{equation}
then the optimal solution $u^\ast$ is simply the parameter that minimizes the mismatch:
\begin{equation}\label{eqn:min}
u^\ast = \text{argmin}_{u}\Phi(u;y)\,.
\end{equation}
This approach however is unable to characterize the uncertainty of the estimation. 
In the Bayesian formulation, one takes a probability point of view, and regards $u$ as a random variable. The aim is to reconstruct the probability distribution of $u$ that combines the prior knowledge and the information from the collected data $y$. More explicitly, let $\rho_\pri(u)$ be the prior distribution, then the posterior distribution of $u$, denoted by $\rho_\pos$, includes the prior distribution, modified by the likelihood function:
\begin{equation}\label{GT}
\rho_\pos(u) =\frac{1}{Z}\exp\left(-\Phi(u;y)\right)\rho_\pri(u)\,.
\end{equation}
The normalization constant $Z$ is given by:
\[
Z:=\int_{\mathcal{X}}\exp\left(-\Phi(u;y)\right)\rho_\pri(u)\rd u\,,\quad\text{so that}\quad \int \rho_\pos(u)\rd u = 1\,.
\]
This perspective provides the full landscape of $u$. While it provides more information, the computational cost is certainly more demanding.

Sampling is one problem emerging under this framework: how to design a cheap numerical solver that generates (hopefully i.i.d.) samples from the target distribution~\eqref{GT}? In particular, suppose one can sample $N$ particles in $\{u^n\}_{n=1}^N\in \mathcal{X}$, and each particle is associated with a weight $w^n$, then how to design the values for $(u^n\,,w^n)$ so that, in some sense
\begin{equation}\label{eqn:approx_goal}
\sum_{n=1}^N w^n\delta_{u^n}\approx \rho_\pos\quad ?
\end{equation}

Many sampling algorithms have been proposed in literature, ranging from classical techniques such as Markov chain Monte Carlo to strategies based on interacting particles. Some set $w^n=\frac{1}{N}$ for all $n$, while others use $u^n$-dependent weights $w^n$. We will explore the latter in this work. 

There are two general guiding principles for designing sampling algorithms: consistency and small variance. 
\begin{itemize}[wide]
\item \textcolor{black}{Consistency means that the ensemble distribution is ``equivalent'' to the target posterior distribution}, in the average sense: When tested on all smooth functions $f$, we require
\begin{equation}\label{eqn:consistent_is}
\EE\left(\sum^N_{n=1}\omega^nf(u^n)\right)=\EE_{\rho_\pos}(f)\,.
\end{equation}
Here the $\EE$ sign on the left hand side means taking expectation of all sampling configurations. Denote the ensemble distribution
\begin{equation}\label{eqn:mu_ensemble_intro}
\mu = \sum_{n=1}^N\omega^n\delta_{u^n}\,,
\end{equation}
then we say $\mu$ is consistent with $\rho_\pos$, or $\mu\sim\rho_\pos$, if \eqref{eqn:consistent_is} holds true. \textcolor{black}{In some literature, this property is called unbiased sampling.}
\item Variance of the weights gives an indicator of the performance of the sampling algorithm, it measures how close each configuration of~\eqref{eqn:mu_ensemble_intro}, from one run of the algorithm, is to the true, i.e. we would like an algorithm so that
\begin{equation}\label{eqn:variance}
\EE\left|\sum^N_{n=1}\omega^nf(u^n)-\EE_{\rho_\pos}(f)\right|^2\quad \text{is small}.
\end{equation}
Once again the $\EE$ sign takes expectation over all possible configurations from the sampling algorithm. For a bounded test function $f$, if $\{(\omega^n,u^n)\}^N_{n=1}$ are i.i.d., then:
\begin{equation}\label{eqn:variance_w}
\begin{aligned}
\EE\left|\sum^N_{n=1}\omega^nf(u^n)-\EE_{\rho_\pos}(f)\right|^2&= \EE\sum^N_{n=1}\left|\omega^nf(u^n)-\frac{1}{N}\EE_{\rho_\pos}(f)\right|^2\\
&= N\EE\left|\left[\omega^1-\frac{1}{N}\right]f(u^1)+\frac{1}{N}f(u^1)-\frac{1}{N}\EE_{\rho_\pos}(f)\right|^2\\
&\leq\frac{2}{N}\mathrm{Var}(N\omega^1)\|f\|^2_{L^\infty}+\frac{2}{N}\EE|f(u^1)-\EE_{\rho_\pos}(f)|^2\\
&\leq\frac{2}{N}\mathrm{Var}(N\omega^1)\|f\|^2_{L^\infty}+\frac{8}{N}\|f\|^2_{L^\infty}\,,
\end{aligned}
\end{equation}
where we use i.i.d. in the second equality and $\EE(N\omega^1)=1$ by consistency. This means the variance of the weight, $\mathrm{Var}(N\omega^1)$, serves as a measure of the performance. If~\eqref{eqn:variance} is small, the algorithm is regarded as a good one. We note that some sampling algorithms cannot provide i.i.d. $\{w^n,u^n\}$ pairs, making the inequality~\eqref{eqn:variance_w} not exactly true. Nevertheless the variance of the weights in some sense quantifies how well each run of the experiment approximates the target posterior distribution.
\end{itemize}

There have been many successful algorithms developed in literature that aim at achieving these two properties. Our algorithms are built upon ideas from some of these methods, including ``Importance Sampling'' (IS) and ``Ensemble Kalman Inversion/Square Root Filter'' (EnKI/EnSRF), all three of which will be briefly recalled below and reviewed in more details in Section~\ref{sec:review}.

Importance Sampling is a rather standard technique: it involves assigning weights to particles so that an easy-to-be-sampled distribution can be turned into the target distribution. The weight is simply the ratio of the two. Regarding the two guiding principles, IS always achieves consistency, but it may give rise to high variance, especially when the easy-to-be-sampled and the target distribution are very different. Some approaches have been proposed to incorporate ``re-sampling" to reduce the variance, such as the strategies used in~\cite{Doucet2001,SMCBOOK,lu2019accelerating}. We do not discuss the details.

EnKI and EnSRF are very different. These two algorithms, both trace the origins to the Kalman filter, require the motion of the particles. They can be seen as the ``analysis step" of data assimilation.  Roughly speaking, the samples are generated from an easy-to-be-sampled distribution, and some dynamics is injected to move the samples around so that after finite time (usually $\text{Time}=1$) they look like i.i.d. samples from the posterior distribution. These two algorithms have completely the opposite properties, compared with IS. There are no weights involved at all, and each particles takes $w^n = \frac{1}{N}$, so the variance of weight is always $0$. However, they are not consistent. This is a disadvantage inherited from Ensemble Kalman Filter: ensemble Kalman filter highly relies on the Gaussianity assumption, that furthermore requires linearity of the forward map -- for nonlinear forward map, the sampling methods are not consistent, in the notion of~\eqref{eqn:consistent_is}.
  
Our goal in this work is to design algorithms by combining advantages of IS and EnKI/EnSRF. We rely on the introduction of the weights to achieve consistency, and the motion introduced in EnKI/EnSRF helps reducing the variance. In this way, we propose  the Weighted-Ensemble-Kalman-Inversion (WEnKI) and Weighted-Ensemble-Square-Root-Filter (WEnSRF) as weighted versions of the EnKI and EnSRF. They achieve consistency for general nonlinear forward maps. We also establish theoretical bounds of the weight variance for the proposed methods. In some sense, this work can be viewed as a correction to EnKI/EnSRF to ensure consistency and an improvement over IS in terms of reducing the weight variance. A natural question then is: how much improvement do we get? As a comparison to EnKI/EnSRF, this amounts to analyzing the strength of the weight term. This is a side product of the paper: by quantifying the differences between EnKI and WEnKI by estimating the weight term, we give a control of the error for EnKI when the forward map in nonlinear.

\textcolor{black}{We should emphasize that besides the two essential properties mentioned above that are theoretically important, there are a lot of practical concerns in implementing algorithms. For example, it would be ideal in real practical problems to design methods that are derivative free, and have low computational complexity. This partially explains the popularity of IS and EnKI. The algorithms are extremely simple, and no derivatives of $\mathcal{G}$ are needed. The proposed new algorithms in this paper fail badly in this dimension: the newly introduced weight terms not only depend on derivatives, but also have very complicated formulation, as will be shown in Section ~\ref{sec:weights}. Although theoretically they are indeed consistent and achieve low variance, such high computational complexity will render them being of little practical use. How to build on top the results obtained in this paper for a practical useful sampling method that also enjoy good theoretical properties will be explored in the near future.}

The rest of this paper is organized in the following: in Section~\ref{sec:review}, we give a brief review of the above mentioned three methods, Important Sampling, Ensemble Kalman Inversion, and Ensemble Square Root Filter. In Section~\ref{sec:weights} we propose our correction to EnKI and EnSRF with added weights. Proof of consistency and some discussion about the control of the variance of weights are presented in Section~\ref{sec:propertyofW}. In Section~\ref{numericalresult} we demonstrate numerical evidence. Some concluding remarks are presented at the end of the paper.

\section{Importance Sampling and ensemble Kalman filter}\label{sec:review}
We review a few sampling strategies in this section. In particular, the Importance Sampling that involves adding weights to the particles to achieve consistency, Ensemble Kalman Inversion and Ensemble Square Root Filter that involve adding motions to the particles so that samples are moved to represent the support of the target.

\subsection{Importance sampling}
The first sampling method we will discuss is the Importance Sampling \cite{IM1989}. It is a fundamental step in Sequential Monte Carlo Methods~\cite{Doucet2001,SMCBOOK}. The idea is extremely simple: one samples a certain amount of particles from the prior distribution, and weight is then calculated based on the ratio of the posterior and the prior evaluation, so the samples with adjusted weights reflect the posterior distribution. The algorithm is summarized in Algorithm~\ref{alg:IP}:
\begin{algorithm}[htb]
\caption{\textbf{Importance sampling}}\label{alg:IP}
\begin{algorithmic}
\State \textbf{Preparation:}
\State 1. Input: $N\gg1$; $\Gamma$; $\mathcal{G}$ (forward map) and $y$ (data).
\State 2. Initial: $\{u^n\}_{n=1}^N$ i.i.d. sampled from the initial distribution $\rho_\text{prior}$.
\State \textbf{Run: } 
1. Calculate the weight, for all $1\leq n\leq N$:
\[
\omega^{n,\ast}=\exp\{-\Phi(u^n;y)\}=\exp\left(-\frac{1}{2}\left|y-\mathcal{G}(u^n)\right|^2_\Gamma\right)\,;
\]

2. Normalize weight:
\[
\omega^n=\frac{\omega^{n,*}}{\sum^N_{n=1}\omega^{n,\ast}}\,.
\]
\State \textbf{Output:} $\{\omega^n\}^N_{n=1},\{u^n\}^N_{n=1}$.
\end{algorithmic}
\end{algorithm}

It is expected that the newly updated distribution is consistent with the target distribution:
\[
\sum_{n=1}^N\omega^n\delta_{u^n}\sim\rho_\pos
\]
in the sense that for any smooth test function $f$:
\begin{equation*}
\EE\left(\sum^N_{n=1}\omega^nf(u^n)\right)=\EE_{\rho_\pos}(f)\,.
\end{equation*}

However, the variance of the weights could be quite large, especially when $\rho_\pos$ and $\rho_\pri$ concentrate at different regions. According to the formulation of the method, this quantity can be explicitly computed:
\begin{equation}\label{chidistance}
\EE((\omega^n)^2)=\frac{1}{N^2}\int_{\mathbb{R}^L} \frac{\rho^2_\pos(u)}{\rho_\pri(u)}\rd u\,,\quad \text{and}\quad \text{Var}(N\omega) = \int_{\mathbb{R}^L} \frac{\rho^2_\pos(u)}{\rho_\pri(u)}\rd u -1\,.
\end{equation}
Thus, if $\rho_\pos$ is non-trivial is the region where $\rho_\pri$ almost vanishes, the quantity can be extremely big, leading to poor performance of the algorithm. Various re-sampling strategies have been proposed~\cite{bain2008fundamentals} to reduce the high variance.

\subsection{Ensemble Kalman filters}
At the other end of the spectrum of sampling method is to not adjust weights at all. Every particle takes equal weight $\frac{1}{N}$. Two typical examples are Ensemble Kalman Inversion and Ensemble Square Root Filter.

The link between sampling and the Kalman filter problem was drawn in an inspiring paper \cite{Reich2011}. Kalman filter (or its more practical version: ensemble Kalman filter) is a class of data assimilation methods that combine data (usually collected at discrete time) with some underlying guessed system dynamics an estimation of parameters in dynamical systems. The dynamics is ran till discrete time when data is collected, and Bayes' rule is applied to update the distribution of the unknown parameters. The paper views the application of the Bayes' rule as an action at a delta function in time, and by inserting a mollifier, the updating process becomes continuous in time.

Such idea was elaborated and formulated into a minimization strategy in \cite{Iglesias_2013}. In~\cite{SS}, the authors view the prior and the (modified) target distribution to be two functions on a function space, and designed a PDE that transforms one to another, either in finite time for the posterior distribution, or in the infinite time horizon for a delta function located at the minimizer. The sampling strategy is in some sense equivalent to the particle method for the PDE: the samples are drawn from the initial distribution, and follow the flow of the PDE by satisfying the associated coupled-ODE/SDE systems. The initial finite-time sampling method is termed ``Ensemble Kalman Inversion (EnKI)'' in~\cite{Iglesias_2013}, and some variations were developed that achieve the final distribution in infinite time, termed ``Ensemble Kalman Sampling (EKS)"~\cite{ding2019ensembleS,EKS,garbunoinigo2019affine}. Since there are no adjustment of weights, the variance of weights keep being $0$ throughout the dynamics. Indeed, upon the well-posedness results of the SDE obtained in~\cite{DCPS,SS}, in~\cite{ding2019ensemble} the authors proved, using the mean-field argument~\cite{CA_IZO_2011,Golse2016}, that when the forward map $\mathcal{G}$ is linear, the method provides approximately i.i.d. samples for the posterior distribution (with $N^{-1/2}$ error in $L_2$-Wasserstein metric).

However, both the derivation of the PDE, and the mean-field limit argument, highly rely on Gaussianity. The forward map is required to be linear for the arguments to carry through. This is not a surprising property since the method was originally derived from Ensemble Kalman Filter and thus inherits its strong requirement: the ``motion" of the particles only depend on the first two moments, and thus the method automatically fails when higher moments are necessary, as in the non-Gaussian case.

We describe both EnKI and EnSRF in details below.

\subsubsection{Ensemble Square Root filter}
The PDE for the ensemble square root filter (EnSRF) writes as the following:
\begin{equation}\label{eqn:ensrf_PDE}
\left\{
\begin{aligned}
&\partial_t\varrho(u,t)-\frac{1}{2}\nabla\cdot\left(\mathrm{Cov}^{\varrho(t)}_{up}\Gamma^{-1}\left(\mathcal{G}(u)+\overline{\MCG}^{\varrho(t)}-2y\right)\varrho\right)=0\,\\
&\varrho(u,0)=\rho_{\pri}
\end{aligned}
\right.\ ,
\end{equation}
where $\overline{\MCG}^{\varrho(t)},\mathrm{Cov}^{\varrho(t)}_{up}$ are expectation of $\mathcal{G}$ and the covariance of $(u,\mathcal{G}(u))$ in $\varrho(u,t)$:
\[
\overline{\MCG}^{\varrho(t)}= \int \mathcal{G}(u)\varrho(t)\rd u\,,\quad \mathrm{Cov}^{\varrho(t)}_{up} =  \int \left(u-\overline{u}\right)\otimes \left(\mathcal{G}(u)-\overline{G}\right)\varrho(t)\rd u\,.
\]
For this particular PDE, one can show that if $\mathcal{G}$ is linear, namely:
\begin{equation}\label{eqn:G_linear}
\mathcal{G}(u) = \mathsf{A}u +\mathsf{b}\,,
\end{equation}
the solution to the PDE~\eqref{eqn:ensrf_PDE} is the target posterior distribution at $t=1$:
\[
\varrho(u,1) =\rho_\pos\,.
\]
Noting that the PDE~\eqref{eqn:ensrf_PDE} is essentially an advection-type PDE, it is easy to formulate the ODE system satisfied by the particles by simply following the trajectory:
\begin{equation}\label{eqn:ODE_srt}
\frac{\rd }{\rd t}u^n_t=-\frac{1}{2}\mathrm{Cov}^{\varrho(t)}_{up}\Gamma^{-1}\left(\mathcal{G}(u^n_t)+\overline{\MCG}^{\varrho(t)}-2y\right)\,,
\end{equation}
with $\{u^n\}$, i.i.d. sampled from $\rho_\pri$ at $t=0$. Since the particles $\{u^n\}$ follow exactly the same flow as the PDE, it is straightforward to have, for $\forall t$:
\begin{equation*}
\frac{1}{N}\sum_j\delta_{u^n(t)}\approx\varrho(u,t)\,.
\end{equation*}
This approximation sign holds true in both weak sense, and in Wasserstein distance sense, for all $t\leq 1$:
\begin{itemize}
\item[--] Weak convergence, for all $f(u)$ bounded continuous :
\[
\EE\left(\int \left(\frac{1}{N}\sum^N_{n=1}\delta_{u^n(t)}-\varrho(u,t)\right) f(u)\rd{u}\right) = 0\,,
\]
and
\[
\EE\left(\int \left(\frac{1}{N}\sum^N_{n=1}\delta_{u^n(t)}-\varrho(u,t)\right) f(u)\rd{u}\right)^2 = \mathcal{O}(N^{-1})\,;
\]
\item[--] Convergence in $L_2$-Wasserstein:
\[
\EE\left(W_2\left(\frac{1}{N}\sum^N_{n=1}\delta_{u^n(t)}\,,\varrho(u,t)\right)\right) \to 0\,.
\]
\end{itemize}
Note that the rate of convergence in $L_2$-Wasserstein depends on the dimension. It is of $\mathcal{O}(N^{-1/2})$ if the dimension of $u$ is smaller than $4$. Details can be found in~\cite{Fournier2015}.

However, in the numerical experiment, since one does not have $\varrho(u,t)$, $\overline{\MCG}$ and $\mathrm{Cov}_{up}(t)$ are not available. In implementation these terms are replaced by the ensemble covariance and the ensemble mean:
\begin{align}\label{eqn:ensemble_variance}
\overline{\MCG}^{\varrho(t)}\to \overline{\MCG}^N(t)=\frac{1}{N}\sum^N_{n=1}\MCG(u^n(t))\,,\ \text{and}\quad 
\overline{u}(t)\to \overline{u}^N(t)=\frac{1}{N}\sum^N_{n=1}u^n(t)\,,
\end{align}
and
\[
\mathrm{Cov}^{\varrho(t)}_{up}\to \mathrm{Cov}^N_{up}(t)=\frac{1}{N}\sum^N_{n=1}\left(u^n(t)-\overline{u}^N(t)\right)\otimes \left(\MCG(u^n(t))-\overline{\MCG}^N(t)\right)\,.
\]
These replacements naturally bring error to realizations of~\eqref{eqn:ODE_srt}. To prove such error is small, the classical mean-field argument is ran. The full recipe of the algorithm is summarized in Algorithm~\ref{alg:ERF}.

\begin{algorithm}[htb]
\caption{\textbf{Ensemble Square Root filter}}\label{alg:ERF}
\begin{algorithmic}
\State \textbf{Preparation:}
\State 1. Input: $N\gg1$; $\Delta t\ll1$ (time step); $M=1/\Delta t$ (stopping index); $\Gamma$;  $\mathcal{G}$ (forward map) and $y$ (data).
\State 2. Initial: $\{u^n_0\}^N_{n=1}$ sampled from initial distribution $\rho_\text{prior}$.

\State \textbf{Run: } Set time step $m=0$;
\State \textbf{While} $m<M$:

1. Define empirical means and covariance:
\[
\overline{u}^N_m=\frac{1}{N}\sum^N_{n=1}u^n_m\,,\quad\overline{\MCG}^N_m=\frac{1}{N}\sum^N_{n=1}\MCG(u^n_m)
\]
and
\[
\mathrm{Cov}^N_{up}=\frac{1}{N}\sum^N_{n=1}\left(u^n_m-\overline{u}^N_m\right)\otimes \left(\MCG(u^n_m)-\overline{\MCG}^N_m\right)\,.
\]

2. Update (set $m\to m+1$)
\begin{equation}\label{eqn:update_ode_srt}
u^n_{m+1}=u^n_m-\frac{\Delta t}{2}\Cov^N_{up}\Gamma^{-1}\left(\mathcal{G}(u^n_{m})+\overline{\MCG}^N_m-2y\right)\,,\quad\forall 1\leq n\leq N\,.
\end{equation}
\State \textbf{end}
\State \textbf{Output:} $\{u^n_M\}^N_{n=1}$.
\end{algorithmic}
\end{algorithm}

It is clear in the algorithm,~\eqref{eqn:update_ode_srt} is simply the forward Euler solver applied on ODE~\eqref{eqn:ODE_srt} with time step being $h=1/M$, and the accuracy would be the standard $\mathcal{O}(h)$. The method was proposed in papers \cite{LIVINGS20081021,Reich2011,ESRF2003} as a data assimilation method. The idea behind the scene is rather simple. Suppose a large number of particles are sampled from a normal distribution $\mathcal{N}(\mu_1,\Sigma_1)$, and to form $\mathcal{N}(\mu_2,\Sigma_2)$, one merely needs to adjust $u^n$ to a new location:
\begin{equation}\label{eqn:gaussian_formulation}
u^n\to \Sigma_2^{1/2}\Sigma_1^{-1/2}(u^n-\mu_1)+\mu_2\,.
\end{equation}
The newly formulated particles are then i.i.d. drawn from $N(\mu_2,\Sigma_2)$. The ODE~\eqref{eqn:ODE_srt} is the continuous in time version of this motion. It is immediate that since only the information of the first two moments is used, Gaussianity is crucial, meaning for consistency, the forward map $\mathcal{G}$ is necessary to be linear.

\subsubsection{Ensemble Kalman Inversion}
A similar approach is used to derive another sampling method called Ensemble Kalman Inversion~\cite{Evensen:2006:DAE:1206873,Reich2011}. The corresponding PDE is the following:
\begin{equation}\label{eqn:eki_PDE}
\left\{
\begin{aligned}
&\partial_t\varrho(u,t)+\nabla_u\cdot\left(\left(y-\mathcal{G}(u)\right)^\top\Gamma^{-1}\mathrm{Cov}^{\varrho(t)}_{pu}\varrho\right)=\frac{1}{2}\mathrm{Tr}\left(\mathrm{Cov}^{\varrho(t)}_{up}\Gamma^{-1}\mathrm{Cov}^{\varrho(t)}_{pu}\mathcal{H}_u\varrho\right)\\
&\varrho(u,0)=\rho_{\pri}
\end{aligned}
\right.\,,
\end{equation}
where $\mathrm{Cov}^{\varrho(t)}_{up}$, and $\mathrm{Cov}^{\varrho(t)}_{pu}$ are covariance of $(u,\mathcal{G})$ and $(\mathcal{G},u)$ in $\varrho(u,t)$. $\mathcal{H}_u\varrho$ is the Hessian of $\varrho$. In~\cite{ding2019ensemble} the authors showed that the solution to the PDE reconstructs the posterior distribution in finite time:
\begin{equation}\label{eqn:reconstruct_eki}
\varrho(t=1,u) = \rho_\pos
\end{equation}
if the forward map $\mathcal{G}$ is linear~\eqref{eqn:G_linear}. So the PDE provides a smooth path to transform the prior distribution to the target in the linear setting.

On the particle level, by following the trajectory of this PDE one has the following SDEs:
\begin{equation}\label{eqn:SDE_eki}
\rd u^n_t=\mathrm{Cov}^{\varrho(t)}_{up}\Gamma^{-1}\left(y-\MCG(u^n_t)\right)dt+\mathrm{Cov}^{\varrho(t)}_{up}\Gamma^{-1/2}\rd W^n_t\,,\quad n = 1\,,\cdots N\,,
\end{equation}
where $dW^n_t$ is the Brownian motion. In the implementation of this SDE, since $\varrho$ is not available, the covariance matrices need to be replaced by the ensemble versions, as is done in~\eqref{eqn:ensemble_variance}, meaning, in the real computation, we use the following coupled SDEs:
\begin{equation}\label{eqn:SDE_eki_couple}
\rd u^n_t=\mathrm{Cov}^N_{up}(t)\Gamma^{-1}\left(y-\MCG(u^n_t)\right)dt+\mathrm{Cov}^N_{up}(t)\Gamma^{-1/2}\rd W^n_t\,,\quad n = 1\,,\cdots N\,,
\end{equation}
where $\mathrm{Cov}^N_{up}$ is the ensemble covariance matrix. Finally we use the ensemble distribution
\[
\frac{1}{N}\sum^N_{n=1}\delta_{u^n(t)}\approx\varrho(u,t)\,,
\]
to approximate the PDE solution.

The discrete version of the coupled SDE~\eqref{eqn:SDE_eki} formulates Algorithm~\ref{alg:EKI}. It is apparent that~\eqref{eqn:update_ujn} is simply the Euler-Maruyama method for~\eqref{eqn:SDE_eki}, as rigorously justified in~\cite{Schilling2018,lange2019continuous}.
\begin{algorithm}[htb]
\caption{\textbf{Ensemble Kalman Inversion}}\label{alg:EKI}
\begin{algorithmic}
\State \textbf{Preparation:}
\State 1. Input: $N\gg1$; $\Delta t\ll1$ (time step); $M=1/\Delta t$ (stopping index); $\Gamma$;  $\mathcal{G}$ (forward map) and $y$ (data).
\State 2. Initial: $\{u^n_0\}^N_{n=1}$ sampled from initial distribution $\rho_\text{prior}$.

\State \textbf{Run: } Set time step $m=0$;
\State \textbf{While} $m<M$:

1. Define empirical means and covariance:
\begin{align}\label{eqn:ensemble_alg}
\overline{u}^N_m=\frac{1}{N}\sum^N_{n=1}u^n_m\,,\quad&\text{and}\quad \mathrm{Cov}^N_{up}=\frac{1}{N}\sum^N_{n=1}\left(u^n_m-\overline{u}^N_m\right)\otimes \left(\MCG(u^n_m)-\overline{\MCG}^N_m\right)\,,\nonumber\\
\overline{\MCG}^N_m=\frac{1}{N}\sum^N_{n=1}\MCG(u^n_m)\,,\quad&\text{and}\quad\mathrm{Cov}^N_{pp}=\frac{1}{N}\sum^N_{n=1}\left(\MCG(u^n_m)-\overline{\MCG}^N_m\right)\otimes \left(\MCG(u^n_m)-\overline{\MCG}^N_m\right)\,.
\end{align}

2. Artificially perturb data (with $\xi^n_{m+1}$ drawn $i.i.d.$ from $\mathcal{N}(0,(\Delta t)^{-1}\Gamma)$):
\begin{equation}\label{eqn:perturbyj}
y^n_{m+1}=y+\xi^n_{m+1},\quad n=1,\dots,N\,.
\end{equation}

3. Update (set $m\to m+1$)
\begin{equation}\label{eqn:update_ujn}
u^n_{m+1}=u^n_m+\mathrm{Cov}^N_{up}\left(\mathrm{Cov}^N_{pp}+(\Delta t)^{-1}\Gamma\right)^{-1}\left(y^n_{m+1}-\MCG(u^n_m)\right)\,,\quad\forall 1\leq n\leq N\,.
\end{equation}
\State \textbf{end}
\State \textbf{Output:} $\{u^n_M\}^N_{n=1}$.
\end{algorithmic}
\end{algorithm}

\textcolor{black}{The EKI method was initially proposed in \cite{Iglesias_2013}, as a further development of~\cite{Reich2011}, to find optimized parameter for inverse problem. Then continuous limit for the discretization in time was considered in \cite{SS} where the authors first wrote down and analyzed the SDE system~\eqref{eqn:SDE_eki}. The well-posedness of this SDE system was shown~\cite{ding2019ensemble}. The wellposedness of the new SDE system~\eqref{eqn:SDE_eki_couple}, which has the covariance replaced by its ensemble version, was proved in~ \cite{DCPS,Schilling2018}. In \cite{ding2019ensemble} the authors showed the mean-field limit of the new coupled SDE system \eqref{eqn:SDE_eki_couple} is the PDE \eqref{eqn:eki_PDE} as $N\to\infty$ (in $L_2$-Wasserstein sense).}

\textcolor{black}{However, we would like to emphasize that in~\cite{ding2019ensemble} it was shown the PDE provides the target distribution only in the linear setting while such convergence holds true for the relatively general weakly nonlinear case. Defending on the perspective, this is in fact a \emph{negative} result for the nonlinear case: \emph{the target distribution is not the solution to the PDE, but the method nevertheless presents the flow to the PDE, so the method does not give a consistent sampling of the target distribution.}}

We also note that often in time, people view EnKI as an optimization algorithm instead of a sampling algorithm, and some relaxation terms have been added for convergence to the minimizer~\cite{chada2019tikhonov,chada2019convergence}.

\subsection{Summary}
It is rather clear that in IS, the particles are kept in the original location, and one merely adjusts the weights. This guarantees the consistency, namely,~\eqref{eqn:consistent_is} always holds true for all bounded continuous functions. On the other hand, since the particles do not move, the weights could be largely suppressed or enlarged, leading to large variance of the weights even in the Gaussian case.

On the contrary, the later two algorithms, EnSRF and EnKI, move particles around to adjust the change of center and the variance. Since all particles are equally weighted, the variance of weight is kept at $0$. However, the derivation of both methods assumes the Gaussianity, and thus the consistency fails for the nonlinear forward map.

\section{Weighted Ensemble Kalman Inversion and Square root filter}\label{sec:weights}
Our proposed algorithms combine the advantages of IS and EnKI/EnSRF, by including both weight and particle dynamics simultaneously, so that we guarantee the consistency at the expense of fairly small variance. The output of the algorithms would be an ensemble distribution having the format of
\begin{equation}\label{eqn:mu_ensemble}
\mathrm{M}_{\text{en}} = \sum_{n=1}^{N}w^n\delta_{u^n}\,,
\end{equation}
as an approximation to the target distribution $\rho_\pos$.

We call the proposed algorithms weighted-EnKI (WEnKI) and weighted-EnSRF (WEnSRF). As the names suggest, we largely keep the format of the flow (or the PDE) for EnKI and EnSRF, while we also add weights to achieve consistency. The underlying flow is designed so that the PDE solution provides a linear interpolation on the log-scale in a time parameter $t$, from the prior to the posterior distributions~\cite{Reich2011, SS}:
\begin{equation}\label{eqn:rho_t}
\rho(u,t) = \frac{1}{Z(t)}\exp\{-t\Phi(u;y)\}\rho_\pri(u)\,, \qquad t \in [0, 1]\,,
\end{equation}
\textcolor{black}{where $\Phi$ is the least-squares function defined in \eqref{eqn:phi}} and $Z(t)$ is a function in time to normalize $\rho(u,t)$ so that
\[
\int\rho(u,t)\rd{u} = 1\,,\qquad \forall t\,.
\]
It is clear that
\[
\rho(u,0) = \rho_\pri\,,\quad\rho(u,1) = \rho_\pos\,,
\]
so the definition~\eqref{eqn:rho_t} provides a flow from the prior to the target posterior distribution. The prior distribution in our algorithm can be quite flexible, for example,
\[
\rho_\pri(u)=\frac{1}{Z}\exp(-V(u))
\]
for a $C^2$ function $V(u)$, with $Z$ being the normalization factor. In practice, however, the prior distribution needs to be an distribution that is easy to sample, so for now we assume:
\begin{equation}\label{eqn:prior}
\rho_\pri= \mathcal{N}(u_0,\Gamma_0)\,.
\end{equation}

The strategy we follow is divided into two steps:
\begin{itemize}
    \item[Step 1:] adjust the PDE~\eqref{eqn:ensrf_PDE} and~\eqref{eqn:eki_PDE} by adding weights so that~\eqref{eqn:rho_t} is a strong solution;
    \item[Step 2:] design a corresponding particle system that carries out the flow of the PDE.
\end{itemize}

Before diving into details of the algorithms, we first introduce some notations. \textcolor{black}{A straightforward but somewhat tedious calculation (see Appendix \ref{Derivetiveofrho}) yields for $\rho(u,t)$ defined in~\eqref{eqn:rho_t}:}
\begin{align}
&\partial_t \rho(u,t)=\left[-\frac{1}{2}\left|y-\mathcal{G}(u)\right|^2_\Gamma+\EE_{\rho(t)}\Bigl(\frac{1}{2}\left|y-\mathcal{G}(u)\right|^2_\Gamma\Bigr)\right]\rho(u,t)\,,\label{partialtmu1}\\
&\nabla \rho(u,t)=\mathcal{V}(u,t)\rho(u,t)\,,\label{partialxmu}\\
&\mathcal{H}_u \rho(u,t)=\left[\mathcal{V}(u,t)\mathcal{V}^\top(u,t)-t\left(\nabla \MCG\right)^\top\Gamma^{-1}\nabla \MCG-\Gamma^{-1}_0+t\mathcal{W}(u)\right]\rho(u,t)\,,\label{partialxxmu}
\end{align}
where $\mathcal{H}_u$ denotes the Hessian with respect to $u$, and  $\mathcal{V}\in\mathbb{R}^{L\times1}$, $\mathcal{W}\in\mathbb{R}^{L\times L}$ are defined as
\begin{align}\label{vu}
& \mathcal{V}(u,t)=t\left(\nabla \MCG(u)\right)^\top\Gamma^{-1}\left(y-\MCG(u)\right)-\Gamma^{-1}_0\left(u-u_0\right)\,, \\
\label{eqn:Wu}
& \mathcal{W}(u)=\left[(\partial_1\nabla\MCG(u))^\top\Gamma^{-1}(y-\MCG(u))\,, \cdots\,, (\partial_L\nabla\MCG(u))^\top\Gamma^{-1}(y-\MCG(u))\right]\,.
\end{align}

\subsection{Weighted ensemble square root filter (WEnSRF)}\label{sec:WEnSRF}
Calculating the left hand side of \eqref{eqn:ensrf_PDE}
using the identities \eqref{partialtmu1}-\eqref{partialxxmu}, we arrive at the PDE that $\rho$, defined in~\eqref{eqn:rho_t}, satisfies:
\begin{equation}\label{eqn:wensrf_PDE}
\partial_t\varrho(u,t)-\frac{1}{2}\nabla\cdot\left(\mathrm{Cov}^{\varrho(t)}_{up}\Gamma^{-1}\left(\mathcal{G}(u)+\overline\MCG^{\varrho(t)}-2y\right)\varrho\right)=\left[\mathcal{P}_1(u,t)+\mathcal{P}_2(u,t)\right]\varrho\,,
\end{equation}
where
\begin{equation}\label{R12}
\begin{aligned}
\mathcal{P}_1(u,t)=&\frac{1}{2}\left(\left|y-\overline\MCG^{\varrho(t)}\right|_{\Gamma}-\left|y-\MCG(u)\right|_{\Gamma}\right)+\frac{1}{2}\mathrm{Tr}\left\{\Cov_{pp}^{\varrho(t)}\Gamma^{-1}\right\}\,,\\
\mathcal{P}_2(u,t)=&-\frac{1}{2}\mathrm{Tr}\left\{\mathrm{Cov}^{\varrho(t)}_{up}\Gamma^{-1}\nabla G(u)\right\}-\frac{1}{2}\mathcal{V}^\top(u,t)\mathrm{Cov}_{up}^{\varrho(t)}\Gamma^{-1}\left(\mathcal{G}(u)+\overline\MCG^{\varrho(t)}-2y\right)
\end{aligned}
\end{equation}
with shorthand notations
\begin{equation}\label{eqn:mean_weights_rho}
\begin{aligned}
&\overline{u}^{\varrho(t)} = \mathbb{E}_{\varrho(t)}(u)\,,\quad\overline{\MCG}^{\varrho(t)} = \mathbb{E}_{\varrho(t)}(\MCG)\,,\\
&\mathrm{Cov}^{\varrho(t)}_{uu} =\mathbb{E}_{\varrho(t)}\left(\left(u-\overline{u}^{\varrho(t)}\right)\otimes \left(u-\overline{u}^{\varrho(t)}\right)\right)\,,\\
&\mathrm{Cov}^{\varrho(t)}_{up} =\mathbb{E}_{\varrho(t)}\left(\left(u-\overline{u}^{\varrho(t)}\right)\otimes\left(\mathcal{G}-\overline{\MCG}^{\varrho(t)}\right)\right)\,,\\
&\mathrm{Cov}^{\varrho(t)}_{pp} =\mathbb{E}_{\varrho(t)}\left(\left(\mathcal{G}-\overline{\MCG}^{\varrho(t)}\right)\otimes\left(\mathcal{G}-\overline{\MCG}^{\varrho(t)}\right)\right)\,.
\end{aligned}
\end{equation}

According to the derivation, it is a natural expectation that $\rho$ is a strong solution. We will further show that the ensemble distribution of particles generated by the sampling method gives a weak solution to the PDE.

The PDE \eqref{eqn:wensrf_PDE} can be solved using standard method of characteristics, which gives arise to the following coupled ODE system for the particles, with $u^n_t$ denoting the location of the $n$-th particle at time $t$, and $w^n_t$ the associated weight:
\begin{equation}\label{eqn:wensrf_SDE}
\left\{
\begin{aligned}
&\rd u^n_t=-\frac{1}{2}\mathrm{Cov}^{\varrho(t)}_{up}\Gamma^{-1}\left(\mathcal{G}(u^n_t)+\overline\MCG^{\varrho(t)}-2y\right)\rd t\\
&\rd w^n_t=\bigl(\mathcal{P}_1(u^n_t,t)+\mathcal{P}_2(u^n_t,t)\bigr)w^n_t\rd t
\end{aligned}
\right.\,.
\end{equation}
The initial condition is chosen so that $\{u^n_0\}^N_{n=1}$ is i.i.d.~sampled from $\rho_\pri(u)\rd u$ and $w^n_0=1/N$, $n = 1, \ldots, N$, to represent initial data $\varrho(u,0) = \rho_\pri$. \textcolor{black}{The system is decoupled, in the sense that when the underlying $\varrho(t)$ is give, each $\{u^n,w^n\}$ pairs are independent from each other.} The output of the algorithm is the empirical distribution:
\begin{equation}\label{eqn:empirical_srf}
\mathrm{M}_{u_t}(u)=\sum^N_{n=1}w^n_t\delta_{u^n_t}\,.
\end{equation}

\textcolor{black}{In practice, however, $\varrho(t)$ is not known and thus $\mathrm{Cov}^{\varrho(t)}$ and $\overline\MCG^{\varrho(t)}$ in~\eqref{eqn:wensrf_SDE} have to be replaced by the ensemble versions:}
\begin{equation}\label{eqn:ensemble_mean}
\overline u^{\varrho(t)}\to\overline u^N(t)=\sum^N_{n=1}w^n(t)u^n(t)\,,\quad \overline\MCG^{\varrho(t)}\to\overline\MCG^N(t)=\sum^N_{n=1}w^n(t)\MCG(u^n(t))\,
\end{equation}
and
\begin{equation}\label{eqn:ensemble_cov}
\mathrm{Cov}^{\varrho(t)}_{up}\to\mathrm{Cov}^N_{up}(t)=\sum^N_{n=1}w^n(t)\left(u^n(t)-\overline{u}^N(t)\right)\otimes \left(\MCG(u^n(t))-\overline{\MCG}^N(t)\right)\,.
\end{equation}

This replacement makes the SDE system tangled up. We summarize the method in Algorithm~\ref{alg:MERF}. Note that due to numerical error, it is typically hard to keep the summation of the weight $1$, and numerically one performs normalization at each time step. Some properties of the method such as the consistency and the boundedness of the variance will be shown in Section~\ref{sec:propertyofW}.

\begin{algorithm}[htb]
\caption{\textbf{Weighted Ensemble Square Root Filter}}\label{alg:MERF}
\setstretch{1}
\begin{algorithmic}
\State \textbf{Preparation:}
\State 1. Input: $N\gg1$; $\Delta t\ll1$ (time step); $M=1/\Delta t$ (stopping index); $\Gamma$; $\mathcal{G}$ (forward map) and $y$ (data).
\State 2. Initial: $\{u^n_0\}^N_{n=1}$ sampled from initial distribution $\rho_\text{prior}$. $\{w^n_0=\frac{1}{N}\}^N_{n=1}$ initial weight.   
\State \textbf{Run: } Set time step $m=0$;
\State \textbf{While} $m<M$:

1. Define empirical means and covariance:
\[
\overline{u}^N_m=\frac{1}{N}\sum^N_{n=1}u^n_m\,,\quad\overline{\MCG}^N_m=\frac{1}{N}\sum^N_{n=1}\MCG(u^n_m),
\]
and
\[
\mathrm{Cov}^N_{up}=\frac{1}{N}\sum^N_{n=1}\left(u^n_m-\overline{u}^N_m\right)\otimes \left(\MCG(u^n_m)-\overline{\MCG}^N_m\right).
\]

2. Update parameters:
\[
\begin{aligned}
\mathcal{V}(u^n_m,t_m)=&t_m\left(\nabla \MCG(u^n_m)\right)^\top\Gamma^{-1}\left(y-\MCG(u^n_m)\right)-\Gamma^{-1}_0\left(u^n_m-u_0\right)\,,\\
\mathcal{P}^n_{m,1}=&\frac{1}{2}\left(\left|y-\overline{\MCG}^N_m\right|_\Gamma-\left|y-\MCG(u^n_m)\right|_\Gamma+\mathrm{Tr}\left\{\Cov_{pp}^N\Gamma^{-1}\right\}\right)\,,\\
\mathcal{P}^n_{m,2}=&-\frac{1}{2}\mathrm{Tr}\left\{\mathrm{Cov}_{up}^N\Gamma^{-1}\nabla G(u^n_m)\right\}\\
&-\frac{1}{2}\mathcal{V}^\top(u^n_m,t_m)\mathrm{Cov}_{up}^N\Gamma^{-1}\left(\mathcal{G}(u^n_m)+\overline{\MCG}^N_m-2y\right)\,.
\end{aligned}
\]

3. Update (set $m\to m+1$): for all $1\leq n\leq N$:
\begin{equation*}\label{eqn:update_ujns}
\begin{aligned}
u^n_{m+1}&=u^n_m-\frac{\Delta t}{2}\Cov^N_{up}\Gamma^{-1}\left(\mathcal{G}(u^n_m)+\overline{\MCG}^N_m-2y\right)\,,\\
w^{n,*}_{m+1}&=w^n_m\exp\left(\Delta t\left(\mathcal{P}^n_{m,1}+\mathcal{P}^n_{m,2}\right)\right)\,,\\
w^n_{m+1}&=\frac{w^{n,*}_{m+1}}{\sum^N_{n=1}w^{n,*}_{m+1}}\,.
\end{aligned}
\end{equation*}
\State \textbf{end}
\State \textbf{Output:} $\{w^n_M\}^N_{n=1}$,$\{u^n_M\}^N_{n=1}$.
\end{algorithmic}
\end{algorithm}

\subsection{Weighted Ensemble Kalman Inversion (WEnKI)}\label{sec:WEnKI}
\textcolor{black}{The same strategy can be applied to modify EnKI to deal with nonlinearity.} Substituting \eqref{partialtmu1}-\eqref{partialxxmu} into \eqref{eqn:eki_PDE}, we have
\begin{equation}\label{eqn:weki_PDE}
\partial_t\varrho(u,t)+\mathcal{L}\left[\varrho\right]=\left[\mathcal{R}_1(u,t)+\mathcal{R}_2(u,t)+\mathcal{R}_3(u,t)\right]\varrho(u,t)\,,
\end{equation}
where $\mathcal{L}$ is a linear operator inherited from~\eqref{eqn:eki_PDE}:
\begin{equation}\label{linearop}
\mathcal{L}\left[\varrho\right]=\nabla_u\cdot\left(\left(y-\mathcal{G}(u)\right)^\top\Gamma^{-1}\mathrm{Cov}^{\varrho(t)}_{pu}\varrho\right)-\frac{1}{2}\mathrm{Tr}\left(\mathrm{Cov}^{\varrho(t)}_{up}\Gamma^{-1}\mathrm{Cov}^{\varrho(t)}_{pu}\mathcal{H}_u(\varrho)\right)\,,
\end{equation}
and the remaining terms $\mathcal{R}_1,\mathcal{R}_2,\mathcal{R}_3$ are given by 
\begin{equation}\label{R1}
\begin{aligned}
\mathcal{R}_1(u,t)=&\frac{1}{2}\mathrm{Tr}\left\{\Cov_{pp}^{\varrho(t)}\Gamma^{-1}-2\left(\nabla \MCG(u)\right)^\top\Gamma^{-1}\Cov_{pu}^{\varrho(t)}\right\}\\
&+\frac{1}{2}\mathrm{Tr}\left\{\Cov_{up}^{\varrho(t)}\Gamma^{-1}\Cov_{pu}^{\varrho(t)}\left[t\left(\nabla \MCG(u)\right)^\top\Gamma^{-1}\nabla \MCG(u)+\Gamma^{-1}_0\right]\right\},\\
\mathcal{R}_2(u,t)=&\frac{1}{2}\left|y-\overline{\MCG}^{\varrho(t)}\right|_\Gamma-\frac{1}{2}\left|y-\MCG(u)-\Cov_{pu}^{\varrho(t)}\mathcal{V}(u,t)\right|_\Gamma\,,\\
\mathcal{R}_3(u,t)=&-\frac{t}{2}\mathrm{Tr}\left\{\mathrm{Cov}_{up}^{\varrho(t)}\Gamma^{-1}\mathrm{Cov}_{pu}^{\varrho(t)}\mathcal{W}(u)\right\}\,,
\end{aligned}
\end{equation}
where $\Cov^{\varrho}_{pp}$, $\Cov^{\varrho}_{up}$, and $\Cov^{\varrho}_{pu}$ are the corresponding covariance matrices, as defined in~\eqref{eqn:mean_weights_rho}. Similar to WEnSRF, we arrive at the following decoupled SDE system:
\begin{equation}\label{eqn:weki_SDE}
\left\{
\begin{aligned}
&\rd u^n_t=\mathrm{Cov}^{\varrho(t)}_{up}\Gamma^{-1}\left(y-\MCG(u^n_t)\right)dt+\mathrm{Cov}^{\varrho(t)}_{up}\Gamma^{-1/2}\rd W^n_t\\
&\rd w^n_t=\bigl(\mathcal{R}_1(u^n_t,t)+\mathcal{R}_2(u^n_t,t)+\mathcal{R}_3(u^n_t,t)\bigr)w^n_t\rd t
\end{aligned}
\right.\,,
\end{equation}
where the Brownian motion is introduced for the second order term in $\mathcal{L}$.
The initial condition is chosen so that $\{u^n_0\}^N_{n=1}$ is i.i.d.~sampled from $\rho_\pri(u)$ and $w^n_0=1/N$, $n = 1, \ldots, N$.

Since $\varrho$ is unknown, as in~\eqref{eqn:ensemble_mean}-\eqref{eqn:ensemble_cov}, we once again replace the true covariance by the ensemble version, and define empirical distribution accordingly:
\begin{equation}\label{ed:weki}
\mathrm{M}_{u_t}(u)=\sum^N_{n=1}w^n_t\delta_{u^n_t}\,.
\end{equation}

There are two sources of randomness involved in WEnKI: the initial sampling and the Brownian motion in~\eqref{eqn:weki_SDE}. Let $\Omega$ be the sample space and $\MCF_0$ be the $\sigma$-algebra: $\sigma\left(u^n(t=0),1\leq n\leq N\right)$, then the filtration is introduced by the dynamics:
\begin{equation*}
\MCF_t=\sigma\left(u^n(t=0),W^n_s,1\leq n\leq N,s\leq t\right)\,.
\end{equation*}
It can be shown the SDE is well-posed in this $\sigma$-algebra~\cite{DCPS,ding2019ensemble}. In next section, we will prove that the empirical distribution is consistent with $\rho(u,t)$ defined in \eqref{eqn:rho_t} under the expectation sense in $\MCF_t$. We will also give control to the variance. The method is summarized in Algorithm~\ref{alg:MEKI}. As in the previous algorithm, the numerical error induces $\sum_nw^n\neq 1$, and an extra renormalization is conducted.

\begin{algorithm}[htb]
\caption{\textbf{Weighted Ensemble Kalman Inversion}}\label{alg:MEKI}
\setstretch{1}
\begin{algorithmic}
\State \textbf{Preparation:}
\State 1. Input: $N\gg1$; $\Delta t\ll1$ (time step); $M=1/\Delta t$ (stopping index); $\Gamma$; $\mathcal{G}$ (forward map) and $y$ (data).
\State 2. Initial: $\{u^n_0\}^N_{n=1}$ sampled from initial distribution $\varrho_\text{prior}$. $\{w^n_0=\frac{1}{N}\}^N_{n=1}$ initial weight.                
\State \textbf{Run: } Set time step $m=0$;
\State \textbf{While} $m<M$:
1. Define empirical means and covariance:
\begin{align*}\label{eqn:ensemble_alg2}
\overline{u}^N_m&=\frac{1}{N}\sum^N_{n=1}w^n_mu^n_m\,,\quad\text{and}\quad\overline{\MCG}^N_m=\frac{1}{N}\sum^N_{n=1}w^n_m\MCG(u^n_m)\,,\nonumber\\
\mathrm{Cov}^N_{pp}&=\frac{1}{N}\sum^N_{n=1}w^n_m\left(\MCG(u^n_m)-\overline{\MCG}^N_m\right)\otimes \left(\MCG(u^n_m)-\overline{\MCG}^N_m\right)\,,\\
\mathrm{Cov}^N_{up}&=\frac{1}{N}\sum^N_{n=1}w^n_m\left(u^n_m-\overline{u}^N_m\right)\otimes \left(\MCG(u^n_m)-\overline{\MCG}^N_m\right)\,.
\end{align*}

2. Define first and second derivative:
\[
\begin{aligned}
\mathcal{V}(u^n_m,t_m)&=t_m\left(\nabla \MCG(u^n_m)\right)^\top\Gamma^{-1}\left(y-\MCG(u^n_m)\right)-\Gamma^{-1}_0\left(u^n_m-u_0\right)\,,\\
\mathcal{W}(u^n_m)&=\left[\partial_1\nabla\MCG(u^n_m)\Gamma^{-1}(y-\MCG(u^n_m))\quad \cdots\quad \partial_L\nabla\MCG(u^n_m)\Gamma^{-1}(y-\MCG(u^n_m))\right]\,.
\end{aligned}
\]

3. Define updated parameter:
\begin{align*}
\mathcal{R}^n_{m,1}=&\frac{1}{2}\mathrm{Tr}\left\{\Cov^N_{pp}\Gamma^{-1}-2\Cov^N_{up}\Gamma^{-1}\nabla \MCG(u^n_m)\right.\\
&\left.+\Cov^N_{up}\Gamma^{-1}\Cov^N_{pu}\left[t_m\left(\nabla \MCG(u^n_m)\right)^\top\Gamma^{-1}\nabla \MCG(u^n_m)+\Gamma^{-1}_0\right]\right\}\,,\\
\mathcal{R}^n_{m,2}=&\frac{1}{2}\left|y-\overline{\MCG}^N_m\right|_\Gamma-\frac{1}{2}\left|y-\MCG(u^n_m)-\Cov^N_{pu}\mathcal{V}(u^n_m,t_m)\right|_\Gamma\,,\\
\mathcal{R}^n_{m,3}=&-\frac{t_m}{2}\mathrm{Tr}\left\{\mathrm{Cov}^N_{up}\Gamma^{-1}\mathrm{Cov}^N_{pu}\mathcal{W}(u^n_{m})\right\}\,.
\end{align*}

4. Artificially perturb data (with $\xi^n_{m+1}$ drawn $i.i.d.$ from $\mathcal{N}(0,(\Delta t)^{-1}\Gamma)$):
\[
y^n_{m+1}=y+\xi^n_{m+1},\quad n=1,\dots,N\,.
\]

5. Update (set $m\to m+1$): for all $n$:
\begin{equation*}\label{eqn:update_ujn2}
\begin{aligned}
u^n_{m+1}&=u^n_m+\mathrm{Cov}^N_{up}\left(\mathrm{Cov}^N_{pp}+(\Delta t)^{-1}\Gamma\right)^{-1}\left(y^n_{m+1}-\MCG(u^n_m)\right)\,,\\
w^{n,*}_{m+1}&=w^n_m\exp\left(\Delta t\left(\mathcal{R}^n_{m,1}+\mathcal{R}^n_{m,2}+\mathcal{R}^n_{m,3}\right)\right)\,,\\
w^n_{m+1}&=\frac{w^{n,*}_{m+1}}{\sum^N_{n=1}w^{n,*}_{m+1}}\,.
\end{aligned}
\end{equation*}
\State \textbf{end}
\State \textbf{Output:} $\{w^n_M\}^N_{n=1}$,$\{u^n_M\}^N_{n=1}$.
\end{algorithmic}
\end{algorithm}

\begin{remark}\label{rmk:WEnKF}
It is important to note that the method is different from running EnKI to time $t=1$ and then apply Important Sampling. The latter was proposed in~\cite{WEnKF} as a weighted version of Ensemble Kalman Filter~\cite{Evensen:2006:DAE:1206873}, known as the Weighted Ensemble Kalman Filter (WEnKF).

Define the conditional mean 
\begin{equation}\label{eqn:wEnkf_mean}
\EE(u^n \mid u^n_0)=u^n_0+\mathrm{Cov}_{up}(u^n_0)\left(\mathrm{Cov}_{pp}(u^n_0)+\Gamma\right)^{-1}\left(y-\MCG(u^n_0)\right)\,,
\end{equation}
and the conditional covariance:
\[
\mathrm{Cov}(u^n \mid u^n_0)=\mathrm{Cov}_{up}(u^n_0)\left(\mathrm{Cov}_{pp}(u^n_0)+\Gamma\right)^{-1}\Gamma\left(\mathrm{Cov}_{pp}(u^n_0)+\Gamma\right)^{-\top}\mathrm{Cov}_{pu}(u^n_0)\,.
\]
In WEnKF, the particle weight is updated according to:
\begin{equation}\label{weightupdatewenkf}
\omega^n=\frac{1}{N}\times\frac{\rho_\pos(u^n)}{\mathcal{N}\left(u^n;\EE(u^n|u^n_0),\mathrm{Cov}(u^n|u^n_0)\right)}\,,
\end{equation}
where $u^n_0$ are the initial samples according to the prior distribution, and $\mathcal{N}$ is the density of Gaussian distribution centered at the conditional mean $\EE(u^n|u^n_0)$. The major difference, compared with the one we propose in~\eqref{eqn:weki_SDE}, is that the covariance used in~\eqref{eqn:wEnkf_mean} is calculated completely from the initial data. The updates along the evolution is entirely ignored. The updating formula in~\eqref{eqn:weki_SDE}, however, involves the weights that evolve in time and is closer to the PDE solution~\eqref{eqn:wensrf_PDE}.
\end{remark}

\textcolor{black}{\begin{remark}
We admit that the added weight terms are quite complicated, both in WEnKI and WEnSRF. In terms of the computational complexity, the new algorithms are far from being ideal. However, if we stick to the flow introduced by EnKI and EnSRF, it seems to difficult to avoid adding some cost to make the algorithm consistent.  Another route to improve the computation is to modify the flow itself, see \cite{reich2019fokkerplanck}. For example, we can involve derivatives in the flow by changing $\Cov_{up}$ to $\Cov_{uu}(\nabla \mathcal{G}(u))^\top$. This potentially could lead to a better flow of the particles, and may potentially provide a less complicated weight term. We leave these to future works. 
\end{remark}
}

\section{Properties of WEnKI and WEnSRF}\label{sec:propertyofW}
We establish a few important properties of WEnKI and WEnSRF in this section. As argued in Section~\ref{sec:review}, the two guiding principles for the algorithm-design is  consistency and  small variance of the weights. These two properties are presented in \S\ref{sec:consistency} and \S\ref{sec:variance} respectively. 
Furthermore, we study the difference between WEnKI and EnKI in \S\ref{sec:diff}, and provide some intuition for EnKI performing well sometimes, even when $\mathcal{G}$ is nonlinear.

\textcolor{black}{We emphasize that in the proof below we use equation \eqref{eqn:wensrf_SDE} and \eqref{eqn:weki_SDE} where the covariance is provided by $\varrho(t)$. In numerics, these covariances matrices need to be replace by their ensemble versions, and another layer of error analysis needs to be added. This is beyond the scope of the current paper.}

\subsection{Consistency}\label{sec:consistency}
The most important property is the consistency, namely, on average, the ensemble mean tested on any smooth function is the same as the real mean. Since the PDEs are obtained by forcing~\eqref{eqn:rho_t} to be the solution, the consistency is expected.

We first present the theorem for WEnSRF.

\begin{theorem}\label{thm:consistent_wensrf}
Assume $\mathcal{G}:\mathbb{R}^L\rightarrow \mathbb{R}^K$ is a $C^1$ function, then:
\begin{itemize}
\item the formula~\eqref{eqn:rho_t} is a strong solution to~\eqref{eqn:wensrf_PDE} with the initial condition $\varrho(u,0)=\rho_\pri$, namely, the PDE~\eqref{eqn:wensrf_PDE} smoothly connects the prior and the posterior distributions;
\item \textcolor{black}{the formula~\eqref{eqn:wensrf_SDE}-\eqref{eqn:empirical_srf} is a weak solution to~\eqref{eqn:wensrf_PDE} with the initial condition $\varrho(u,0)=\mathrm{M}_{u_0}(u)$, namely, the ODE system~\eqref{eqn:wensrf_SDE} follows the flow of the transition: for any smooth test function $f:\RR^L\rightarrow \RR$ and $0\leq t\leq 1$, we have consistency:}
\begin{equation}\label{EEmatch3}
\EE_{\rho(t)}(f)=\EE\left(\sum^N_{n=1}w^n_tf(u^n_t)\right)=\EE\left(\EE_{\mathrm{M}_{u_t}}(f)\right)\,,
\end{equation}
where the $\EE$ on the outer layer of the right hand side comes from the random configuration of the initial condition for $\{u^n_0\}$.
\end{itemize}
\end{theorem}
\begin{proof}
The first point is trivial: it amounts to substituting the solution~\eqref{eqn:rho_t} into the equation and balancing terms. To show that the empirical measure $\mathrm{M}_{u^n}$ is the weak solution to the PDE, we test it with a smooth function $f(u)$. Note that
\[
\EE_{\mathrm{M}_{u_t}}(f) = \int \sum_{n=1}^Nw^n\delta_{u^n}f(u) \rd{u} = \sum^N_{n=1}w^n_tf(u^n_t)\,,
\]
we have
\begin{equation*}
\begin{aligned}
\frac{\rd}{\rd t}\EE_{\mathrm{M}_{u_t}}(f)=&\frac{\rd}{\rd t}\sum^N_{n=1}w^n_tf(u^n_t)=\sum^N_{n=1} \frac{\rd w^n_t}{\rd t}f(u^n_t)+w^n_t\frac{\rd f(u^n_t)}{\rd t}\\
=&\sum^N_{n=1}\left[\mathcal{P}_1(t,u^n_t)+\mathcal{P}_2(t,u^n_t)\right]w^n_tf(u^n_t)\\
&-\frac{1}{2}w^n_t\left(\nabla f(u^n_t)\right)^\top\mathrm{Cov}^{\varrho(t)}_{up}\Gamma^{-1}\left(\mathcal{G}(u^n_t)+\overline\MCG-2y\right)\\
=&\EE_{\mathrm{M}_{u_t}}\left(\left[\mathcal{P}_1(u,t)+\mathcal{P}_2(u,t)\right]f(u)\right.\\
&-\EE_{\mathrm{M}_{u_t}}\left(\frac{1}{2}\left(\nabla f(u)\right)^\top\mathrm{Cov}^{\varrho(t)}_{up}\Gamma^{-1}\left(\mathcal{G}(u)+\overline\MCG-2y\right)\right)\,.
\end{aligned}
\end{equation*}
This is exactly the weak formulation of~\eqref{eqn:wensrf_PDE} tested on $f$ with the integration by parts applied on the advection term. To show~\eqref{EEmatch3}, we simply note that both $\rho$ and $\varrho = \mathrm{M}_{u_t}$ are weak solutions.
\end{proof}

The same type of theorem holds true for WEnKI:

\begin{theorem}\label{thm:consistent_wenki}
\textcolor{black}{If $\mathcal{G}:\mathbb{R}^L\rightarrow \mathbb{R}^K$ is $C^2$ and Lipschitz continuous,}
\begin{itemize}
\item the formula~\eqref{eqn:rho_t} is a strong solution to~\eqref{eqn:wensrf_PDE} with the initial condition $\rho(u,0)=\rho_\pri$, namely, the PDE~\eqref{eqn:weki_PDE} characterizes the dynamics in~\eqref{eqn:rho_t} and connects the prior and the posterior distributions;

\item the formula \eqref{eqn:weki_SDE}-\eqref{ed:weki}, in expectation, is a weak solution to~\eqref{eqn:wensrf_PDE} with the initial condition  $\rho(u,0)=\mathrm{M}_{u_0}(u)$, namely, the SDE system~\eqref{eqn:weki_SDE} follows the flow of the transition: for any smooth test function $f:\RR^L\rightarrow \RR$ and $0\leq t\leq 1$,
\begin{equation}\label{EEmatch}
\EE_{\rho(t)}(f)=\EE\left(\sum^N_{n=1}w^n_tf(u^n_t)\right)=\EE\left(\EE_{\mathrm{M}_{u_t}}(f)\right)\,,
\end{equation}
where the outer-layer $\mathbb{E}$ on the right hand side is taken in the probability space $\left(\Omega,\MCF_t,\mathbb{P}\right)$. 
\end{itemize}
\end{theorem}

\begin{proof} The first part is again trivial. 
To show~\eqref{EEmatch}, we first realize that the equation holds trivially for $t=0$ since $\{u^n_0\}^N_{n=1}$ are i.i.d sampled from $\rho_\pri(u)$. \textcolor{black}{For all $t>0$, we plug in~\eqref{eqn:weki_SDE} and apply the It\^o's formula on
\[
\rd\sum^N_{n=1}w^n_tf(u^n_t)=\sum^N_{n=1}\rd w^n_tf(u^n_t)+w^n_t\rd f(u^n_t)\,
\]
and use
\[
\begin{aligned}
\rd f(u^n_t)=&\left(\nabla f(u^n_t)\right)^\top\mathrm{Cov}^{\varrho(t)}_{up}\Gamma^{-1}\left(y-\MCG(u^n_t)\right)\rd t+\left(\nabla f(u^n_t)\right)^\top\mathrm{Cov}^{\varrho(t)}_{up}\Gamma^{-1/2}dW^n_t\\
&+\frac{1}{2}\mathrm{Tr}\left\{\mathcal{H}_v(f(u^n_t))\mathrm{Cov}^{\varrho(t)}_{up}\Gamma^{-1}\mathrm{Cov}^{\varrho(t)}_{pu}\right\}\rd t\\
\end{aligned}
\]
to get
\begin{equation*}
\begin{aligned}
\rd\EE\bigl( \EE_{\mathrm{M}_{u_t}}(f)\bigr)&={\rd\EE\left(\sum^N_{n=1}w^n_tf(u^n_t)\right)}\\
&=\EE\sum^N_{n=1}\left(\left[\mathcal{R}_1(t,u^n_t)+\mathcal{R}_2(t,u^n_t)+\mathcal{R}_3(t,u^n_t)\right]w^n_tf(u^n_t)\rd t\right.\\
&\quad \left.+w^n_t\left(\nabla f(u^n_t)\right)^\top\mathrm{Cov}^{\varrho(t)}_{up}\Gamma^{-1}\left(y-\MCG(u^n_t)\right)\rd t\right.\\
&\quad \left.+\frac{1}{2}w^n_t\mathrm{Tr}\left\{\mathcal{H}_v(f(u^n_t))\mathrm{Cov}^{\varrho(t)}_{up}\Gamma^{-1}\mathrm{Cov}^{\varrho(t)}_{pu}\right\}\rd t\right)\\
&=\EE\left(\EE_{\mathrm{M}_{u_t}}\left(\left[\mathcal{R}_1+\mathcal{R}_2+\mathcal{R}_3\right]f(u)\right)\rd t\right)\\
&\quad +\EE\left(\EE_{\mathrm{M}_{u_t}}\left(\left(\nabla f(u)\right)^\top\mathrm{Cov}^{\varrho(t)}_{up}\Gamma^{-1}\left(y-\MCG(u)\right)\rd t\right.\right.\\
&\quad +\left.\left.\frac{1}{2}\mathrm{Tr}\left\{\mathcal{H}_v(f(u))\mathrm{Cov}^{\varrho(t)}_{up}\Gamma^{-1}\mathrm{Cov}^{\varrho(t)}_{pu}\right\}\rd t\right)\right)\,.
\end{aligned}
\end{equation*}
In the second equation above, we used the fact that the SDE that $u_t$ satisfies is wellposed due to the Lipschitz continuity of $\mathcal{G}$.}

This formulation is exactly the weak formulation of~\eqref{eqn:weki_PDE} tested on $f$ with integration by parts moving the $\nabla$ and $\mathcal{H}$ onto $f$. The equality~\eqref{EEmatch} follows as $\rho(u,t)$ defined in~\eqref{eqn:rho_t} is also a weak solution.
\end{proof}

\subsection{Bounding the variance of weights}\label{sec:variance} 
We investigate the behavior of the weights for a fairly  large class of $\mathcal{G}$. Throughout this subsection, we will impose one of the two assumptions on the forward map $\mathcal{G}$ below.

The first assumption is rather weak, and it only requires the boundedness of derivatives of $\mathcal{G}$ up to second order.

\begin{assum}\label{Gassum1}
$\MCG:\mathbb{R}^L\to\mathbb{R}^K$ is $C^2$ function and there exists $\Lambda>0$ such that
\begin{equation}\label{eqn:assme_G_Lip}
\|\nabla \MCG\|_2\leq \Lambda,\quad \|\mathcal{H}\left(|\MCG|^2_{\Gamma}\right)\|_2\leq \Lambda\quad \|\partial_i\nabla \MCG\|_2\leq \Lambda,\quad 1\leq i\leq L\,.
\end{equation}
\end{assum}

The second assumption is slightly stronger, and it asks for the structure of the range of the linear and nonlinear components of $\mathcal{G}$.
\begin{assum}\label{Gassum2}
$\mathcal{G}$ is weakly-nonlinear in the sense that there exists a matrix $\mathsf{A}\in\mathcal{L}(\mathbb{R}^L,\mathbb{R}^K)$ so that
\begin{equation}\label{linear}
\mathcal{G}(u)=\mathsf{A}u+\Rm(u)\,,
\end{equation}
where $\Rm(u)$ is a $C^2$ bounded Lipschitz function from $\mathbb{R}^L$ to $\mathbb{R}^K$ satisfying
\begin{equation*}
\Gamma^{-1/2}\Rm(u)\perp \Gamma^{-1/2}\mathsf{A}u,\quad \forall u\in \mathbb{R}^L\,,
\end{equation*}
and there exists constants $\Lambda,\Lambda_1$ and $M$ such that: for $1\leq i\leq L$
\begin{equation}\label{eqn:assum2}
\|\nabla \Rm\|_2\leq \Lambda_1\leq\Lambda,\quad|\Rm|\leq M,\quad \|\mathsf{A}\|_2\leq \Lambda,\quad  \|\mathcal{H}\left(|\MCG|^2_{\Gamma}\right)\|_2\leq \Lambda,\quad \|\partial_i\nabla \MCG\|_2\leq \Lambda\,.
\end{equation}
\end{assum}

If the second assumption holds true, we call the optimal solution for the linear part:
\begin{equation}\label{eqn:udagger}
u^\ast_\mathsf{A} =\min_{u}\|y-\mathsf{A}u\|_{\Gamma}\,,
\end{equation}
and the associated residue
\begin{equation}\label{eqn:errorsfr}
\mathsf{r}=y-\mathsf{A}u^\ast_{\mathsf{A}}\,.
\end{equation}
It is then automatic that
\begin{equation}\label{r:orthorgonal}
\Gamma^{-1/2}\mathsf{r}\perp \Gamma^{-1/2}\mathsf{A}u,\quad \forall u\in \mathbb{R}^L\,.
\end{equation}
We also define the Gaussian part of the distribution
\begin{equation}\label{eqn:rho_A}
\rho_{\mathsf{A}}(u,t)=\frac{1}{Z(t)}\exp\left(-\frac{t}{2}|\mathsf{A}u^\ast_{\mathsf{A}}-\mathsf{A}u|^2_{\Gamma}-\frac{1}{2}|u-u_0|^2_{\Gamma_0}\right)\,,
\end{equation}
so that we have 
\[
\rho(u,t)\propto \rho_\mathsf{A}(u,t)\exp\left(-\frac{t}{2}|\mathsf{r}-\Rm(u)|^2_{\Gamma}\right)\,.
\]
This $\rho_{\mathsf{A}}$ has expectation and the covariance matrix:
\begin{equation}\label{eqn:mean_variance}
\begin{aligned}
&u_\mathsf{A}(t)=\left(t\mathsf{A}^\top \Gamma^{-1}\mathsf{A}+\Gamma^{-1}_0\right)^{-1}\left(t\mathsf{A}^\top \Gamma^{-1}\mathsf{A}u^\ast_{\mathsf{A}}+\Gamma^{-1}_0u_0\right)\,\text{,}\\
&\mathrm{Cov}_\mathsf{A}(t)=\left(t\mathsf{A}^\top \Gamma^{-1}\mathsf{A}+\Gamma^{-1}_0\right)^{-1}\,.
\end{aligned}
\end{equation}

It is immediate that the second assumption is stronger than the first one, and thus one would expect a tighter bound. Indeed, by comparing Theorem~\ref{thm:loosebound} and Theorem~\ref{thm:betterestimation}, we see that the variance of weights is bounded by a constant that exponentially grows with respect to $|y|$, the data, when only Assumption~\ref{Gassum1} holds true, but is bounded by a constant independent of $|y|$ when Assumption~\ref{Gassum2} also holds.

Since EnKI is a more popular method than EnSRF, the analysis is conducted on WEnKI mainly. Similar analysis could potentially be applied to deal with WEnSRF but could be more delicate. We do not pursue it in this paper.

We also note that the analysis is conducted on the dynamics~\eqref{eqn:weki_SDE} with coefficients calculated from the exact density, and thus each particle is evolved independently (this is in the spirit of the McKean-Vlasov dynamics or propagation of chaos, expected in the mean field limit). The analysis for the numerical version, with all the covariance matrices replaced by the ensemble ones as seen in~\eqref{eqn:ensemble_mean}-\eqref{eqn:ensemble_cov} will be left for future works.

\subsubsection{Bounded nonlinearity under Assumption~\ref{Gassum1}}
We first prove a lemma to bound covariance matrix of $\varrho(u,t)$.
\begin{lemma}\label{weightlemma}
Under Assumption \ref{Gassum1},
\[
\EE^{\varrho(t)}\left( |y-\MCG|^2_{\Gamma}\right)
\]
decreases in $t$, where $\varrho(u,t)$ is the solution to~\eqref{eqn:weki_PDE}.
\end{lemma}
\begin{proof}
By Theorem~\ref{thm:consistent_wenki}, $\varrho = \rho(u,t)$ defined in~\eqref{eqn:rho_t} is a strong solution to the PDE. Taking partial derivative with respect to $t$ and rewriting~\eqref{partialtmu1}, we get 
\begin{equation}\label{partialtmu}
\partial_t\varrho=-\frac{1}{2}\left\{\left| y - \MCG\right|^2_{\Gamma}-\EE^{\varrho(t)}\left(\left| y - \MCG\right|^2_{\Gamma}\right)\right\}\varrho.
\end{equation}
Multiplying $\lvert y - \MCG\rvert^2_{\Gamma}$ on both sides and taking integral yields
\[
\frac{\rd}{\rd t}\EE^{\varrho(t)}\left| y - \MCG\right|^2_{\Gamma}= -\frac{1}{2}\left(\EE^{\varrho(t)}\left(\left| y - \MCG\right|^4_{\Gamma}\right) - \left(\EE^{\varrho(t)}\left(\left| y - \MCG\right|^2_{\Gamma}\right)\right)^2\right) \leq 0\,,
\]
which concludes the lemma.
\end{proof}

\begin{lemma}\label{lem:2}
Under Assumption~\ref{Gassum1}, there exists a finite constant $C$ depending on $\Lambda$, $|u_0|$, $\|\Gamma_0^{-1}\|_2$, $\|\Gamma^{-1}\|_2$ and $|y|$ only such that for $0\leq t\leq 1$:
\begin{equation}\label{lossb1}
\EE^{\varrho(t)}|u|^2< C\,,\quad \EE^{\varrho(t)}|\MCG(u)|^2< C\,.
\end{equation}
\end{lemma}

In the proof below, we use $C$ to denote a generic constant that changes from line to line, and we keep track of the constant's dependence on different argument. However, we do not specify the form of the dependence.

\begin{proof}
Consider
\[
\varrho(u,0) = \rho_\pri = \exp\left(-|u-u_0|^2_{\Gamma_0}\right)\,,
\]
we expand $\mathcal{G}$ around $\vec{0}$ and utilize the bound~\eqref{eqn:assme_G_Lip} for:
\[
\EE^{\varrho(0)}|\MCG(u)|^2\leq \Lambda^2\EE^{\varrho(0)}|u|^2+|\MCG(\vec{0})|^2\leq \Lambda^2(|u_0|^2+\textrm{Tr}(\Gamma_0))+|\MCG(\vec{0})|^2\,.
\]
Therefore,
\begin{equation}\label{eq:varrho0bound}
\begin{aligned}
    \EE^{\varrho(0)}\left(\left| y - \MCG\right|^2_{\Gamma}\right)&\leq 2\|\Gamma^{-1}\|_2\left(|y|^2+\EE^{\varrho(0)}|\MCG(u)|^2\right)\\
&\leq 2\|\Gamma^{-1}\|_2\left(|y|^2+\Lambda^2(|u_0|^2+\textrm{Tr}(\Gamma_0))+|\MCG(\vec{0})|^2\right) \\
& =: C_1 |y|^2 + C_2 \,,
\end{aligned}
\end{equation}
where the last line defines constants $C_1$ and $C_2$, which only depend on $\Lambda$, $|u_0|$, $\|\Gamma_0^{-1}\|_2$, $\|\Gamma^{-1}\|_2$.

Multiplying $|\MCG(u)|^2$ and $|u|^2$ on both sides of \eqref{partialtmu}, we get
\begin{equation*}
\begin{aligned}
\frac{\rd}{\rd t}\EE^{\varrho(t)}|\MCG(u)|^2&=-\frac{1}{2}\int\left\{\left| y - \MCG\right|^2_{\Gamma}-\EE^{\varrho(t)}\left(\left| y - \MCG\right|^2_{\Gamma}\right)\right\}|\mathcal{G}(u)|^2\rho(t) \rd u\\
&\leq \frac{1}{2} \int \EE^{\varrho(t)}\left(\left| y - \MCG\right|^2_{\Gamma}\right)|\mathcal{G}(u)|^2\rho(t) \rd u \\
& = \frac{1}{2} \EE^{\varrho(t)}\left(\left| y - \MCG\right|^2_{\Gamma}\right)\EE^{\varrho(t)}|\MCG(u)|^2\\
&\leq \frac{1}{2} \EE^{\varrho(0)}\left(\left| y - \MCG\right|^2_{\Gamma}\right)\EE^{\varrho(t)}|\MCG(u)|^2\\
&\stackrel{\eqref{eq:varrho0bound}}{\leq}(C_1|y|^2 + C_2)\EE^{\varrho(t)}|\MCG(u)|^2\,,
\end{aligned}
\end{equation*}
and
\begin{equation*}
\begin{aligned}
\frac{\rd}{\rd t}\EE^{\varrho(t)}|u|^2&=-\frac{1}{2}\int\left\{\left| y - \MCG\right|^2_{\Gamma}-\EE^{\varrho(t)}\left(\left| y - \MCG\right|^2_{\Gamma}\right)\right\}|u|^2\rho \rd u\\
&\leq \frac{1}{2}\EE^{\varrho(t)}\left(\left| y - \MCG\right|^2_{\Gamma}\right)\EE^{\varrho(t)}|u|^2\\
&\leq \frac{1}{2} \EE^{\varrho(0)}\left(\left| y - \MCG\right|^2_{\Gamma}\right)\EE^{\varrho(t)}|u|^2\\
&\stackrel{\eqref{eq:varrho0bound}}{\leq}(C_1|y|^2 + C_2)\EE^{\varrho(t)}|u|^2\,,
\end{aligned}
\end{equation*}
where we use Lemma \ref{weightlemma} in the second inequalities. By Gr\"onwall inequality, we have:
\[
\EE^{\varrho(t)}|u|^2\leq  \EE^{\varrho(0)}|u|^2e^{(C_1|y|^2+C_2)t}\leq \left(|u_0|^2+\mathrm{Tr}(\Gamma_0)\right)e^{\left(C_1|y|^2+C_2\right)t}\,,
\]
and
\begin{equation}\label{eqn:variance_const_exp}
\begin{aligned}
\EE^{\varrho(t)}|\MCG(u)|^2&\leq  \EE^{\varrho(0)}|\MCG(u)|^2e^{(C_1|y|^2+C_2)t}\\
&\leq \left(\Lambda^2(|u_0|^2+\textrm{Tr}(\Gamma_0))+|\MCG(\vec{0})|^2\right)e^{\left(C_1|y|^2+C_2\right)t}\,.
\end{aligned}
\end{equation}
Choose $C$ to be the bigger value of the two with $t=1$, we conclude the lemma.
\end{proof}

The immediate consequence of Lemma \ref{weightlemma} and Lemma~\ref{lem:2} is the boundedness of the covariance matrices:
\begin{cor}\label{col:bound_Cov}
Under Assumption \ref{Gassum1}, there exists a constant $C$ depending on $\Lambda$, $|u_0|$, $\|\Gamma_0^{-1}\|_2$, $\|\Gamma^{-1}\|_2$ and $|y|$ such that for $0\leq t\leq 1$
\begin{equation}\label{lossb2}
\|\Cov^{\varrho(t)}_{uu}\|_2\leq C\,,\quad\|\Cov^{\varrho(t)}_{up}\|_2\leq C\,,\quad \|\Cov^{\varrho(t)}_{pp}\|_2\leq C\,,
\end{equation}
where $\Cov^{\varrho(t)}_{uu},\Cov^{\varrho(t)}_{up},\Cov^{\varrho(t)}_{pp}$ are the corresponding covariance matrices, as defined in \eqref{eqn:mean_weights_rho}.
\end{cor}

These \emph{a priori} estimates are now used to bound the variance of the weights.
\begin{theorem}\label{thm:loosebound}
Under Assumptions \ref{Gassum1}, let $\{u^n_t,\omega^n_t\}^N_{n=1}$ solve \eqref{eqn:weki_SDE}. Then there exists a constant $C$ only depending on $\Lambda$, $|u_0|$, $\|\Gamma_0^{-1}\|_2$, $\|\Gamma^{-1}\|_2$ and $|y|$ such that for any $0\leq t\leq 1$. 
\begin{equation*}
\mathrm{Var}(N\omega^n_t) \leq C\,.
\end{equation*}
\end{theorem}

\begin{remark}
We note that the result in Theorem~\ref{thm:loosebound} is not optimal. The constant, if traced carefully, blows up as $|y|\to\infty$ with a rate of at least $e^{|y|^2}$, as suggested in~\eqref{eqn:variance_const_exp}. Essentially this result does not demonstrate WEnKI superior than the classical IS. However, as will be shown in Theorem~\ref{thm:betterestimation}, under a stronger assumption (Assumption~\ref{Gassum2}), the dependence on $y$ could be removed.
\end{remark}

\begin{proof}
Note that
\[
\mathrm{Var}(N\omega^n_t) = \EE(N\omega_t^n-1)^2 = N^2\left(\EE|\omega^n_t|^2-\frac{1}{N^2}\right)\,,
\]
thus to prove the theorem, it suffices to show that 
\begin{equation}\label{omegajsquare}
N^2\EE|\omega^n_t|^2 \leq C\,
\end{equation}
with $C$ depending on $\Lambda$, $|u_0|$, $\|\Gamma_0^{-1}\|_2$, $\|\Gamma^{-1}\|_2$ and $|y|$.

Multiplying $\omega^n$ on both sides of the second equation of \eqref{eqn:weki_SDE} and taking expectation, we have 
\begin{equation}\label{omegainequality}
\begin{aligned}
\frac{\rd}{\rd t}\EE|\omega^n|^2 &\leq 2\EE\left\{\left[\mathcal{R}_1(u^n_t,t)+\mathcal{R}_2(u^n_t,t)+\mathcal{R}_3(u^n_t,t)\right]|w^n_t|^2\right\}\\
&\leq 2\left(\|\mathcal{R}_1\|_{\infty}+\frac{1}{2}\left(y-\overline{\MCG}^{\varrho}\right)^\top\Gamma^{-1}\left(y-\overline{\MCG}^{\varrho}\right)+\|\mathcal{R}_3\|_{\infty}\right)\EE|\omega^n|^2,
\end{aligned}.
\end{equation}
where we have omitted the last three terms in $\mathcal{R}_2$ because the sum of them is negative. We then bound the three terms in bracket separately. As a preparation, we note that
\[
\mathrm{Tr}\left\{\Cov^{\varrho(t)}_{pp}\Gamma^{-1}\right\}=\int (\MCG(u)-\overline{\MCG})^\top\Gamma^{-1}(\MCG(u)-\overline{\MCG}) \varrho(u,t)\rd u\leq \EE^{\varrho(t)}|\MCG(u)-\overline{G}|^2\|\Gamma^{-1}\|_2\,,
\]
and
\[
\begin{aligned}
\mathrm{Tr}\left\{\Cov^{\varrho(t)}_{up}\Gamma^{-1}\right\}&=\int (u-\overline{u})^\top \Gamma^{-1}(\MCG(u)-\overline{\MCG}) \varrho(u,t)\rd u\\
&\leq \left(\EE^{\varrho(t)}|u-\overline{u}||\MCG(u)-\overline{\MCG}|\right)\|\Gamma^{-1}\|_2\,.
\end{aligned}
\] 
Apply these inequalities to estimate $\mathcal{R}_k$ defined in~\eqref{R1}, we arrive at the following bounds.

\begin{equation}\label{boundr1}
\begin{aligned}
&|\mathcal{R}_1(u,t)|\\
\leq &\frac{\|\Gamma^{-1}\|_2}{2}\Biggl\{\EE^{\varrho(t)}|\MCG(u)-\overline{\MCG}^{\varrho(t)}|^2+\left[2\Lambda+\|\Cov^{\varrho(t)}_{up}\|_2\left(t\|\Gamma^{-1}\|_2\Lambda^2+\|\Gamma^{-1}_0\|_2\right)\right] \\
& \hspace{8em} \times \left(\EE^{\varrho(t)}|u-\overline{u}^{\varrho(t)}||\MCG(u)-\overline{\MCG}^{\varrho(t)}|\right)\Biggr\}\\
\leq &\frac{\|\Gamma^{-1}\|_2}{2}\Bigl\{\mathrm{Tr}(\Cov^{\varrho(t)}_{pp})+\left[2\Lambda+\|\Cov^{\varrho(t)}_{up}\|_2\left(t\|\Gamma^{-1}\|_2\Lambda^2+\|\Gamma^{-1}_0\|_2\right)\right] \\
&\hspace{8em} \times \mathrm{Tr}(\Cov^{\varrho(t)}_{pp})^{1/2}\mathrm{Tr}(\Cov^{\varrho(t)}_{uu})^{1/2}\Bigr\}\\
\leq &C\,,
\end{aligned}
\end{equation}
where the last inequality comes from Corollary~\ref{col:bound_Cov}.

For the non-negative contribution from $\mathcal{R}_2$, we have 
\begin{equation}\label{boundr2}
\begin{aligned}
\left(y-\overline{\MCG}^{\varrho(t)}\right)^\top\Gamma^{-1}\left(y-\overline{\MCG}^{\varrho(t)}\right)&\leq \|\Gamma^{-1}\|_2\left|y-\overline{\MCG}^{\varrho(t)}\right|^2\\
&\leq 2\|\Gamma^{-1}\|_2\left(|y|^2+\EE^{\varrho(t)}|\MCG(u)|^2\right)\leq C\,,
\end{aligned}
\end{equation}
where the last inequality comes from Lemma~\ref{lem:2}.

Finally, for $\mathcal{R}_3$, we have 
\begin{equation}\label{boundr3}
\begin{aligned}
|\mathcal{R}_3(u,t)|&\leq \frac{1}{2} t\|\Gamma^{-1}\|_2\|\Cov_{up}^{\varrho(t)}\|_2\|\mathcal{W}(u)\|_2 \left(\EE^{\varrho(t)}|u-\overline{u}^{\varrho(t)}||\MCG(u)-\overline{\MCG}^{\varrho(t)}|\right)\\
&\leq \frac{1}{2} t\|\Gamma^{-1}\|_2\|\Cov_{up}^{\varrho(t)}\|_2\|\mathcal{W}(u)\|_2\mathrm{Tr}(\Cov^{\varrho(t)}_{pp})^{1/2}\mathrm{Tr}(\Cov^{\varrho(t)}_{uu})^{1/2}\\
&\leq C\,,
\end{aligned}
\end{equation}
where we have used Corollary~\ref{col:bound_Cov}, and that, by definition of $\mathcal{W}$,
\begin{equation}\label{boundforwu}
\begin{aligned}
\|\mathcal{W}(u)\|_2&\leq \|\mathcal{W}(u)\|_F\\
&\leq\frac{L}{2}\left(\|\mathcal{H}\left(|\MCG|^2_{\Gamma}\right)\|_2+\|\Gamma^{-1}\|_2\|\nabla \MCG\|^2_2\right)+|y|\|\Gamma^{-1}\|_2\max_{1\leq i\leq L}\{\|\partial_i\nabla \MCG\|_2\}\leq C\,. 
\end{aligned}
\end{equation}

All the constants above depend on $\Lambda$, $|u_0|$, $\|\Gamma_0^{-1}\|_2$, $\|\Gamma^{-1}\|_2$ and $|y|$. Substitute these into \eqref{omegainequality}, we have
\[
\rd \EE|\omega^n_t|^2\leq C\EE|\omega^n_t|^2\,.
\]
Realizing that $\omega^n_0=\frac{1}{N}$ so that $\EE|\omega^n_0|^2=\frac{1}{N^2}$, we obtain 
\[
\EE|\omega^n_t|^2\leq \frac{e^{Ct}}{N^2}\,.
\]
This concludes~\eqref{omegajsquare} and this theorem.
\end{proof}

\subsubsection{Weak nonlinearity under Assumption~\ref{Gassum2}} 
The variance bound can be improved when we assume further structure of the nonlinearity, namely, when the nonlinear component $\Rm(u)$ is perpendicular to the range of the linear component $\mathsf{A}$, weighted by $\Gamma^{-1/2}$. In particular, the bound becomes independent of $y$, as shown in the following theorem. 
\begin{theorem}\label{thm:betterestimation}
Under Assumption \ref{Gassum2}, there exists a finite constant $C$ depending on $\Lambda_1$, $\Lambda$, $\|\Gamma^{-1}\|_2$, $\|\Gamma^{-1}_0\|_2$, $|\mathsf{r}|$, 
$M$, and $|u^\ast_{\mathsf{A}}|$, such that
\begin{equation}\label{estimate:bettervarththm}
\|\mathrm{Var}(N\omega^n_t)\|_{{L^\infty}[0,1]}\leq C\,.
\end{equation}
Furthermore,
\begin{equation}\label{eqn:ass2_estimate_bound}
\lim_{\Lambda_1\rightarrow0} C\leq C_1\,,
\end{equation}
where $C_1$ only depends on $\Lambda$, $\|\Gamma_0^{-1}\|_2$, $\|\Gamma^{-1}\|_2$.
\end{theorem}

\begin{remark}
This theorem is a counterpart of Theorem~\ref{thm:loosebound}, but stronger assumption on the nonlinearity is added. As a result, the variance of weight is bounded, independent of $y$. In the most extreme case, suppose $\mathcal{G}$ is entirely linear, $\Lambda_1=0$, then according to the theorem, the variance is bounded by a fixed constant. As a comparison, if one applies Important Sampling directly, for large $y$ and thus large $u^\ast$, the variance blows up at the order of $\mathcal{O}(e^{|u^\ast|^2})$, equivalently to $\mathcal{O}(e^{|y|^2})$ for reasonably conditioned $\mathsf{A}$. This means that under mild conditions (Assumption~\ref{Gassum2}), the newly proposed WEnKI method significantly reduces the weight variance from the classical method IS.
\end{remark}

The proof of the theorem is largely based on the following calculation.
\begin{proposition}\label{thm:newweightbound} 
Under Assumption \ref{Gassum2}, let $\{u^n_t,\omega^n_t\}^N_{n=1}$ solve \eqref{eqn:weki_SDE}, we have
\begin{equation}\label{betterboundforvarianceofweight}
\frac{\rd }{\rd t}\left(
\begin{aligned}
& \EE|u^n_t|^2(N\omega^n_t)^2\\
&\mathrm{Var}(N\omega^n_t)+1
\end{aligned}
\right) \leq CW(t)\left(
\begin{aligned}
& \EE|u^n_t|^2(N\omega^n_t)^2\\
&\mathrm{Var}(N\omega^n_t)+1
\end{aligned}
\right),\quad \forall 0\leq t\leq 1.
\end{equation}
where $W(t)$ is a $2\times 2$ matrix defined by
\begin{align*}
W_{1,1}(t) & =C\left[(\mathrm{Var}^{\rho(t)}(u))^2+\mathrm{Var}^{\rho(t)}(u)+|u^\ast_{\mathsf{A}}|\mathrm{Var}^{\rho(t)}(u)+\left|u^\ast_{\mathsf{A}}-\overline{u}^{\rho(t)}\right|^2+1\right]\,,\\
W_{1,2}(t) & =C|\mathrm{Var}^{\rho(t)}(u)|\left[\mathrm{Var}^{\rho(t)}(u)+\left|u^\ast_{\mathsf{A}}\right|+1\right]\,, \\
W_{2,1}(t)&=C\left(\left|\overline{u}^{\rho(t)}-u^\ast_{\mathsf{A}}\right|\left\|I-(\Cov_\mathsf{A})^{-1}\Cov^{\rho(t)}_{u,u}\right\|_2+\|\Cov^{\rho(t)}_{\Rm,u}\|_2\right)\,,\\
\intertext{and}
W_{2,2}(t)&=C\left[\left|\overline{u}^{\rho(t)}-u^\ast_{\mathsf{A}}\right|\left(\left|\overline{u}^{\rho(t)}-\Cov^{\rho(t)}_{u,u}(\Cov_\mathsf{A}(t))^{-1}u_\mathsf{A}\right|\right.\right.\\
&\qquad+\left.\left\|I-(\Cov_\mathsf{A})^{-1}\Cov^{\rho(t)}_{u,u}\right\|_2+\|\Cov_{u,u}^{\rho(t)}\|_2\Lambda_1\right)\\
&\qquad\left.+(\mathrm{Var}^{\rho(t)}(u))^2+\mathrm{Var}^{\rho(t)}(u)\right]+\|\Cov^{\rho(t)}_{\Rm,u}\|_2(|u_\mathsf{A}|+\Lambda_1+1)\,,
\end{align*}
where $\mathrm{Var}^{\rho(t)}(u)=\mathrm{Tr}\left(\Cov^{\rho(t)}_{u,u}\right)$ and $C$ is a constant depending on $\Lambda$, $\|\Gamma^{-1}\|_2$, $\|\Gamma^{-1}_0\|_2$, $|\mathsf{r}|$, and $M$.
\end{proposition}

The proof for the proposition is deferred to Appendix \ref{Appendix1}. We now give the proof for Theorem~\ref{thm:betterestimation} based on the above proposition.

\begin{proof}[Proof of Theorem~\ref{thm:betterestimation}]
For fixed $1\leq n\leq N$, let
\[
p(t)=\EE|u^n_t|^2(N\omega^n_t)^2,\quad q(t)=\mathrm{Var}(N\omega^n_t)+1\,.
\]
Since $\omega^n_0 =\frac{1}{N}$,
\[
p(0)=\EE_{\rho_{\pri}}|u|^2,\quad q(0)=1\,.
\]
According to Proposition~\ref{thm:newweightbound},
\[
\frac{\rd}{\rd t}\left(
\begin{aligned}
&p(t)\\
&q(t)
\end{aligned}
\right)\leq W(t)\left(
\begin{aligned}
&p(t)\\
&q(t)
\end{aligned}
\right)\,,
\]
which implies \eqref{estimate:bettervarththm}. If $\Lambda_1\rightarrow0$, nonlinear function $\Rm(u)$ is almost a constant. Therefore, we also have $\overline{u}^{\rho(t)}\rightarrow u_\mathsf{A}(t)$, $\Cov^{\rho(t)}_{u,u}\rightarrow\Cov_\mathsf{A}(t)$ and $\|\Cov^{\rho(t)}_{\Rm,u}\|_2\rightarrow0$, then the coefficients for $q$ satisfy:
\[
\lim_{\Lambda_1\rightarrow0} W_{2,1}(t)=0,\quad \lim_{\Lambda_1\rightarrow0} W_{2,2}(t)=(\mathrm{Tr}\left(\Cov_\mathsf{A}\right))^2+\mathrm{Tr}\left(\Cov_\mathsf{A}\right)\,.
\]
Then~\eqref{eqn:ass2_estimate_bound} is a direct consequence, concluding the theorem.
\end{proof}

\subsection{EnKI with nonlinear forward map}\label{sec:diff}
In this section, we study a slightly different topic: how different are WEnKI and EnKI?
In fact, it was proved in~\cite{ding2019ensemble} that EnKI is not a consistent sampling method when the forward map is nonlinear. The algorithm, without the weight, can be regarded as the discrete version of PDE~\eqref{eqn:eki_PDE}, but the target distribution $\rho(u,t)$ is not the solution to the PDE, and hence EnKI is inconsistent. 

It is numerically observed, however, that despite being inconsistent, EnKI mysteriously performs rather well~\cite{Reich2011}, especially when the target distribution is almost Gaussian-like, no matter how nonlinear $\mathcal{G}$ is, also see the book~\cite{reich_cotter_2015} for more examples. To the best of our knowledge, such discrepancy in terms of theoretical and practical performance,  has not been addressed in literature. In this subsection, as a first attempt to explain it, we provide one criterion, under which, EnKI performs similarly well as WEnKI.

The argument in the end comes down to comparing the continuous version of WEnKI and EnKI, two Fokker-Planck equations, with the former one having a weight term while the latter not.

Once again we denote $\rho$ the target distribution, defined in~\eqref{eqn:rho_t} and proved to be the solution to equation~\eqref{eqn:weki_PDE} in Theorem~\ref{thm:consistent_wenki}, and let $\varrho$ the solution to the Fokker-Planck equation without the weight:
\begin{equation}\label{eqn:middleeki_PDEthm1}
\left\{
\begin{aligned}
&\partial_t\varrho(u,t)+\nabla_u\cdot\left(\left(y-\mathcal{G}(u)\right)^\top\Gamma^{-1}\mathrm{Cov}^{\varrho(t)}_{pu}\varrho\right)=\frac{1}{2}\mathrm{Tr}\left(\mathrm{Cov}^{\varrho(t)}_{up}\Gamma^{-1}\mathrm{Cov}^{\varrho(t)}_{pu}\mathcal{H}_u\varrho\right)\\
&\varrho(u,0)=\rho_{\pri}
\end{aligned}
\right.\ ,
\end{equation}
where $\mathrm{Cov}^{\varrho(t)}_{up}$, and $\mathrm{Cov}^{\varrho(t)}_{pu}$ are covariance of $(u,\mathcal{G})$ and $(\mathcal{G},u)$ in $\varrho(u,t)$. It was proved in~\cite{ding2019ensemble} that \eqref{eqn:middleeki_PDEthm1} is the mean-field limit of EnKI.

We will now show that $\rho$ and $\varrho$ are close when the weight term (defined in~\eqref{R1})
\[
\mathcal{W}(u,t)=\mathcal{R}_1(u,t)+\mathcal{R}_2(u,t)+\mathcal{R}_3(u,t)
\]
is small. This means that WEnKI and EnKI give more or less the same results when the weight term is small. We recall the  bounded Lipschitz metric $(d_{BL})$ between probability measures:
\[
d_{BL}(\mu,\nu)=\sup_{f\in\mathrm{Lip}(\mathbb{R}^L)}\left|\int_{\mathbb{R}^L}fd\mu-\int_{\mathbb{R}^L}fd\nu\right|\,,    
\]
where 
\[
\mathrm{Lip}(\mathbb{R}^L)=\Bigl\{f\in \mathrm{C}_b\,:\,\sup_x |f(x)|\leq1,\,\sup_{x\neq y}\frac{|f(x)-f(y)}{|x-y|}\leq 1\Bigr\}\,.
\]
Since the admissible set in the supremum is smaller than the class of Lipschitz-$1$ function and $1$-bounded function, this metric can be bounded by $L^2$-Wasserstein distance $W_2(\mu,\nu)$ and total variation $\mathrm{TV}(\mu,\nu)$
\begin{equation}\label{boundofnorm}
d_{BL}(\mu,\nu)\leq W_2(\mu,\nu),\quad d_{BL}(\mu,\nu)\leq \mathrm{TV}(\mu,\nu)\,.
\end{equation}

We have the following theorem characterizing the difference between $\varrho$ and $\rho$, \textit{i.e.}, EnKI and WEnKI (that is consistent to the target distribution).
\begin{theorem}\label{EKibound}
Under Assumption \ref{Gassum1}, there exists a constant $C$ depending on $\Lambda$, $|u_0|$, $\|\Gamma^{-1}_0\|_2$, $\|\Gamma^{-1}\|_2$, $|y|$, such that
\begin{equation}
d_{BL}(\varrho(u,t) \rd u,\rho(u,t) \rd u)\leq C\int^1_0\int (1+|u|^2)|\mathcal{W}|\rho\rd u\rd s\,
\end{equation}
for all $0\leq t\leq 1$.
\end{theorem}

This theorem states that the size of the weight gives control over the distance between $\rho$ and $\varrho$. To compare them, we introduce an intermediate surrogate $\widetilde{\rho}$, given by 
\begin{equation}\label{eqn:middleeki_PDE}
\left\{
\begin{aligned}
&\partial_t\widetilde{\rho}(u,t)+\nabla_u\cdot\left(\left(y-\mathcal{G}(u)\right)^\top\Gamma^{-1}\mathrm{Cov}^{\rho(t)}_{pu}\widetilde{\rho}\right)=\frac{1}{2}\mathrm{Tr}\left(\mathrm{Cov}^{\rho(t)}_{up}\Gamma^{-1}\mathrm{Cov}^{\rho(t)}_{pu}\mathcal{H}_u\widetilde{\rho}\right)\\
&\widetilde{\rho}(u,0)=\rho_{\pri}
\end{aligned}
\right.\ ,
\end{equation}
where $\mathrm{Cov}^{\rho(t)}_{up}$ and $\mathrm{Cov}^{\rho(t)}_{pu}$ are given by $\rho(u,t)$. We will bound $d_{BL}(\rho, \widetilde{\rho})$ and $d_{BL}(\widetilde{\rho}, \varrho)$ in the following two propositions. The theorem is a direct consequence of the two. 

\begin{proposition}\label{prop:stability}
Under Assumption \ref{Gassum1},  there exists a constant $C$ depending on $\Lambda$, $|u_0|$, $\|\Gamma^{-1}_0\|_2$, $\|\Gamma^{-1}\|_2$ and $|y|$, such that 
\begin{equation}\label{thm:stability}
d_{BL}(\widetilde{\rho}(u,t) \rd u,\rho(u,t) \rd u)\leq \mathrm{TV}(\widetilde{\rho}(u,t) \rd u,\rho(u,t) \rd u)\leq C\int^1_0\int |\mathcal{W}|\rho\rd u\rd s\,
\end{equation}
for all $0\leq t\leq 1$.
\end{proposition}

\begin{proposition}\label{prop:stability2}
Under Assumption~\ref{Gassum1}, there exists a constant $C$ depending on $\Lambda$, $|u_0|$, $\|\Gamma^{-1}_0\|_2$, $\|\Gamma^{-1}\|_2$ and $|y|$, such that 
\begin{equation}\label{thm:wass:stability}
d_{BL}(\varrho(u,t) \rd u,\widetilde{\rho}(u,t) \rd u)\leq W_2(\varrho(u,t)\rd u,\widetilde{\rho}(u,t) \rd u)\leq C\int^1_0\int (1+|u|^2)|\mathcal{W}|\rho \rd u\rd s\,
\end{equation}
for all $0\leq t\leq 1$.
\end{proposition}

\begin{proof}[Proof of Proposition~\ref{prop:stability}] 
The proof is based on the following construction of particle system. Let
\begin{equation}\label{eqn:weki_SDE2}
\left\{
\begin{aligned}
&\rd u_t=\mathrm{Cov}_{up}^{\rho(t)}\Gamma^{-1}\left(y-\MCG(u_t)\right)\rd t+\mathrm{Cov}_{up}^{\rho(t)}\Gamma^{-1/2}\rd W_t\,\\
&\rd w_t=\mathcal{W}(u,t)w_t\rd t
\end{aligned}\,,
\right.
\end{equation}
with initial data $u_0$ sampled from $\mu_\pri$ and $w_0=1$. This is a Langevin dynamics, so that for any test function $f$:
\[
\EE(f(u_t))=\EE_{\widetilde{\rho}(t)}f,\quad \EE(w_tf(u_t))=\EE_{\rho(t)}f\,.
\]

It is clear from second equality in \eqref{eqn:weki_SDE2}:
\[
w_t>0
\]
for all $t$, and that
\[
\rd|w_t-1|\leq |\rd w_t-1|\leq \left|\mathcal{W}(u_t,t)\right|w_t\rd t\,.
\]
This means
\[
\frac{\rd}{\rd t} \EE|w_t-1|\leq \EE\left(\left|\mathcal{W}(u_t,t)\right|w_t\right)=\int \left|\mathcal{W}\right|\rho \rd u
\]
and
\[
\EE|w_t-1|\leq \int^1_0\int \left|\mathcal{W}\right|\rho \rd u \rd t,\quad \forall 0\leq t\leq 1.
\]
The boundedness~\eqref{thm:stability} is a direct result by using $L^\infty$ test function to bound total variation:
\begin{equation*}
\begin{aligned}
\mathrm{TV}(\widetilde{\rho}\rd u,\rho \rd u)&\leq\left|\sup_{\|f\|_\infty=1}\int f(\widetilde{\rho}-\rho)\rd u\right|\leq \sup_{\|f\|_\infty=1}\left|\EE (w_t-1)f(u_t)\right|\\
&\leq \sup_{\|f\|_\infty=1}\|f\|_\infty \EE|w_t-1|\\
&\leq \int^1_0\int \left|\mathcal{W}\right|\rho \rd u \rd s\,.
\end{aligned}
\end{equation*}
\end{proof}

\begin{proof}[Proof of Proposition~\ref{prop:stability2}]

We first state and prove an estimate of the difference of covariance
\begin{equation}\label{stability2}
\|\Cov^{\rho(t)}_{u,p}-\Cov^{\widetilde{\rho}(t)}_{u,p}\|_2\leq C\int^1_0\int (1+|u|^2)|\mathcal{W}|\rho \rd u\rd s
\end{equation}
for all $0 \leq t \leq 1$. For this, we first bound $\EE |u_t|^2 |w_t-1|$ using It\^o's formula and \eqref{eqn:weki_SDE2}:
\[
\rd |u_t|^2 |w_t-1|=2\left\langle \rd u_t,u_t\right\rangle\left|w_t-1\right|+\left\langle \rd u_t,\rd u_t\right\rangle\left|w_t-1\right|+|u_t|^2\rd |w_t-1|\,.
\]
Taking expectation on both sides, we have
\[
\begin{aligned}
\rd \EE|u_t|^2 |w_t-1|\leq &\EE\left\langle\mathrm{Cov}_{up}^{\rho(t)}\Gamma^{-1}\left(y-\MCG(u_t)\right),u_t\right\rangle|w_t-1|\rd t+C\EE|w_t-1|\rd t\\
&+\EE\left[|u_t|^2\left|\mathcal{W}(u_t,t)\right|w_t\right]\rd t\\
\leq &C\EE \left[(|u_t|^2+|u_t|)|w_t-1|\right]\rd t+C\int^1_0\int (1+|u|^2)\left|\mathcal{W}\right|\rho \rd u \rd s\\
\leq &C\EE \left[|u_t|^2|w_t-1|\right]+C\left(\EE |w_t-1|\right)^{1/2}(\EE |u_t|^2|w_t-1|)^{1/2}\\
&+C\int^1_0\int (1+|u|^2)\left|\mathcal{W}\right|\rho \rd u \rd s\\
\leq &C\EE |u_t|^2|w_t-1|\\
&+C\left(\int^1_0\int (1+|u|^2)\left|\mathcal{W}\right|\rho \rd u \rd t\right)^{1/2}(\EE |u_t|^2|w_t-1|)^{1/2}\\
&+C\int^1_0\int (1+|u|^2)\left|\mathcal{W}\right|\rho \rd u \rd s\,,
\end{aligned}
\]
where we use Corollary \ref{col:bound_Cov} and equation~\eqref{lossb2}.

By Gr\"onwall's inequality and $w_0=1$, we get
\begin{equation}\label{2momentstability}
\left\|\int u\otimes u (\widetilde{\rho}-\rho)\rd u\right\|_2\leq\EE\left[|u_t|^2 |w_t-1|\right]\leq C\int^1_0\int (1+|u|^2)\left|\mathcal{W}\right|\rho \rd u \rd s
\end{equation}
and
\begin{equation}\label{1momentstability}
\begin{aligned}
\left|\int |u|(\widetilde{\rho}-\rho)\rd u\right|&\leq\EE|u_t| |w_t-1|\leq \left(\EE |w_t-1|\right)^{1/2}(\EE |u_t|^2|w_t-1|)^{1/2}\\
&\leq C\int^1_0\int (1+|u|^2)\left|\mathcal{W}\right|\rho \rd u \rd s
\end{aligned}
\end{equation}
for any $t\leq 1$. Combining \eqref{2momentstability} and \eqref{1momentstability}, we have
\[
\begin{aligned}
\|\Cov^{\rho(t)}_{u,p}-\Cov^{\widetilde{\rho}(t)}_{u,p}\|_2&\leq \left\|\int u\otimes u (\widetilde{\rho}-\rho)\rd u\right\|_2+\left|\int |u|(\widetilde{\rho}-\rho)\rd u\right|\left|\int |u|(\widetilde{\rho}+\rho)\rd u\right|\\
&\leq C\int^1_0\int (1+|u|^2)\left|\mathcal{W}\right|\rho \rd u \rd s\,,
\end{aligned}
\]
which proves \eqref{stability2}.

\medskip 

We now come back to the Proposition to prove \eqref{thm:wass:stability}. We use two particle systems to represent \eqref{eqn:middleeki_PDE} and \eqref{eqn:middleeki_PDEthm1}. Let
\begin{equation}\label{eqn:stabilityweki_SDE}
\rd u_t=\mathrm{Cov}^{\rho(t)}_{up}\Gamma^{-1}\left(y-\MCG(u_t)\right)\rd t+\mathrm{Cov}_{up}^{\rho(t)}\Gamma^{-1/2}\rd W_t\,,
\end{equation}
where the initial data $u_0$ is sampled from $\rho_\pri(u)$, and let
\begin{equation}\label{eqn:stabilityweki_SDE2}
\rd v_t=\mathrm{Cov}^{\varrho(t)}_{up}\Gamma^{-1}\left(y-\MCG(v_t)\right)\rd t+\mathrm{Cov}_{up}^{\varrho(t)}\Gamma^{-1/2}\rd W_t\,
\end{equation}
with the same initial data $v_0=u_0$. Then immediately
\[
W_2(\widetilde{\rho},\varrho)\leq \left(\EE|u_t-v_t|^2\right)^{1/2}\,.
\]
To show the theorem, it suffices to prove
\begin{equation}\label{wass:stability2}
\EE|u_t-v_t|^2\leq C\sup_{t \in [0, 1]} \|\mathrm{Cov}_{up}^{\rho(t)}-\mathrm{Cov}_{up}^{\widetilde{\rho}(t)}\|_2\,
\end{equation}
for all $t$ and then utilize~\eqref{stability2}.

Let $\gamma_t=u_t-v_t$, one subtracts \eqref{eqn:stabilityweki_SDE2} from \eqref{eqn:stabilityweki_SDE} and uses It\^o's formula to obtain
\begin{equation*}
\begin{aligned}
\frac{\rd}{\rd t}\EE|\gamma_t|^2& \leq \EE\left\langle \mathrm{Cov}^\rho_{up}\Gamma^{-1}\left(y-\MCG(u_t)\right)-\mathrm{Cov}^{\varrho}_{up}\Gamma^{-1}\left(y-\MCG(v_t)\right),\gamma_t\right\rangle \rd t\\
&\qquad +\frac{1}{2}\mathrm{Tr}\left\{\left(\mathrm{Cov}_{up}^{\rho(t)}-\mathrm{Cov}_{up}^{\varrho(t)}\right)\Gamma^{-1}\left(\mathrm{Cov}_{up}^{\rho(t)}-\mathrm{Cov}_{up}^{\varrho(t)}\right)\right\}\rd t\\
& =\EE\left\langle \left(\mathrm{Cov}^{\rho(t)}_{up}-\mathrm{Cov}^{\varrho(t)}_{up}\right)\Gamma^{-1}\left(y-\MCG(u_t)\right),\gamma_t\right\rangle\rd t\\
&\quad -\EE\left\langle\mathrm{Cov}^{\varrho(t)}_{up}\Gamma^{-1}\left(\MCG(u_t)-\MCG(v_t)\right),\gamma_t\right\rangle \rd t\\
&\quad +\frac{1}{2}\mathrm{Tr}\left\{\left(\mathrm{Cov}_{up}^{\rho(t)}-\mathrm{Cov}_{up}^{\varrho(t)}\right)\Gamma^{-1}\left(\mathrm{Cov}_{up}^{\rho(t)}-\mathrm{Cov}_{up}^{\varrho(t)}\right)\right\}\rd t\\
& \leq C\|\mathrm{Cov}_{up}^{\rho(t)}-\mathrm{Cov}_{up}^{\varrho(t)}\|_2(|y|^2+\EE|u_t|^2)^{1/2}(\EE|\gamma_t|^2)^{1/2}+C\|\mathrm{Cov}_{up}^{\varrho(t)}\|_2\EE|\gamma_t|^2\\
&\quad +C\|\mathrm{Cov}_{up}^{\rho(t)}-\mathrm{Cov}_{up}^{\varrho(t)}\|^2_2\,.
\end{aligned}
\end{equation*}
Since
\[
\|\mathrm{Cov}_{up}^{\widetilde{\rho}(t)}-\mathrm{Cov}_{up}^{\varrho(t)}\|_2\leq (\EE|\gamma_t|^2)^{1/2},
\]
we have
\[
\begin{aligned}
\|\mathrm{Cov}_{up}^{\rho(t)}-\mathrm{Cov}_{up}^{\varrho(t)}\|_2&\leq \|\mathrm{Cov}_{up}^{\rho(t)}-\mathrm{Cov}_{up}^{\widetilde{\rho}(t)}\|_2+\|\mathrm{Cov}_{up}^{\widetilde{\rho}(t)}-\mathrm{Cov}_{up}^{\varrho(t)}\|_2\\&\leq \|\mathrm{Cov}_{up}^{\rho(t)}-\mathrm{Cov}_{up}^{\widetilde{\rho}(t)}\|_2+C(\EE|\gamma_t|^2)^{1/2}\,.
\end{aligned}
\]

Therefore
\begin{equation*}
\begin{aligned}
\frac{\rd}{\rd t}\EE|\gamma_t|^2\leq C\EE|\gamma_t|^2+C\|\mathrm{Cov}_{up}^{\rho(t)}-\mathrm{Cov}_{up}^{\widetilde{\rho}(t)}\|_2(\EE|\gamma_t|^2)^{1/2}+\|\mathrm{Cov}_{up}^{\rho(t)}-\mathrm{Cov}_{up}^{\widetilde{\rho}(t)}\|^2_2\,.
\end{aligned}
\end{equation*}
Since $\gamma_0=0$, by Gr\"onwall's inequality, we finally arrive at 
\[
\EE|\gamma_t|^2\leq C\|\|\mathrm{Cov}_{up}^{\rho(t)}-\mathrm{Cov}_{up}^{\widetilde{\rho}(t)}\|_2\|_{L^\infty_{[0,1]}}
\]
for all $t<1$, which proves~\eqref{wass:stability2}, concluding the proposition.
\end{proof}

\section{Numerical results}\label{numericalresult}
In this section, we show some numerical evidence to demonstrate the superiority of the proposed method. All numerical examples are highly nonlinear, so we are away from the known ``safe zone''  where EnKI and EnSRF work both perfect. We remark that we conduct numerical experiments only in low dimensional setting to have a clear illustration of the behavior of the algorithm.

\subsection{One dimension example}
As a start, we first test out the 1D case. We set the normal distribution $\mathcal{N}(0,1)$ as the prior distribution.

\begin{enumerate}[label=$\bullet$, wide]
\item Example 1: In this example we set $\mathcal{G}(u)=4\cos(2(u-3))+\sin(u-3)$ and the data (with only one observation) is given at $y=0$. The posterior distribution is a multimodal distributions, as shown in Figure~\ref{Figure1}. The number of samples is set to be $N=1000$, and in WEnKI and WEnSRF, we choose the time step $\Delta t=10^{-3}$. As a comparison, we plot the result using WEnKF (Remark~\ref{rmk:WEnKF}) and Important Sampling, EnKI and EnSRF. In this example, the prior and the posterior distributions share supports, so IS and WEnKF, the two methods that achieve consistency, behave relatively well. But due to nonlinearity and non-Gaussianity, EnKI and EnSRF, the two methods that tend to give one-mode Gaussian-like profile, fail.
\begin{figure}[!ht]
     \centering
     \subfloat{\includegraphics[width=1\textwidth,height=0.4\textheight]{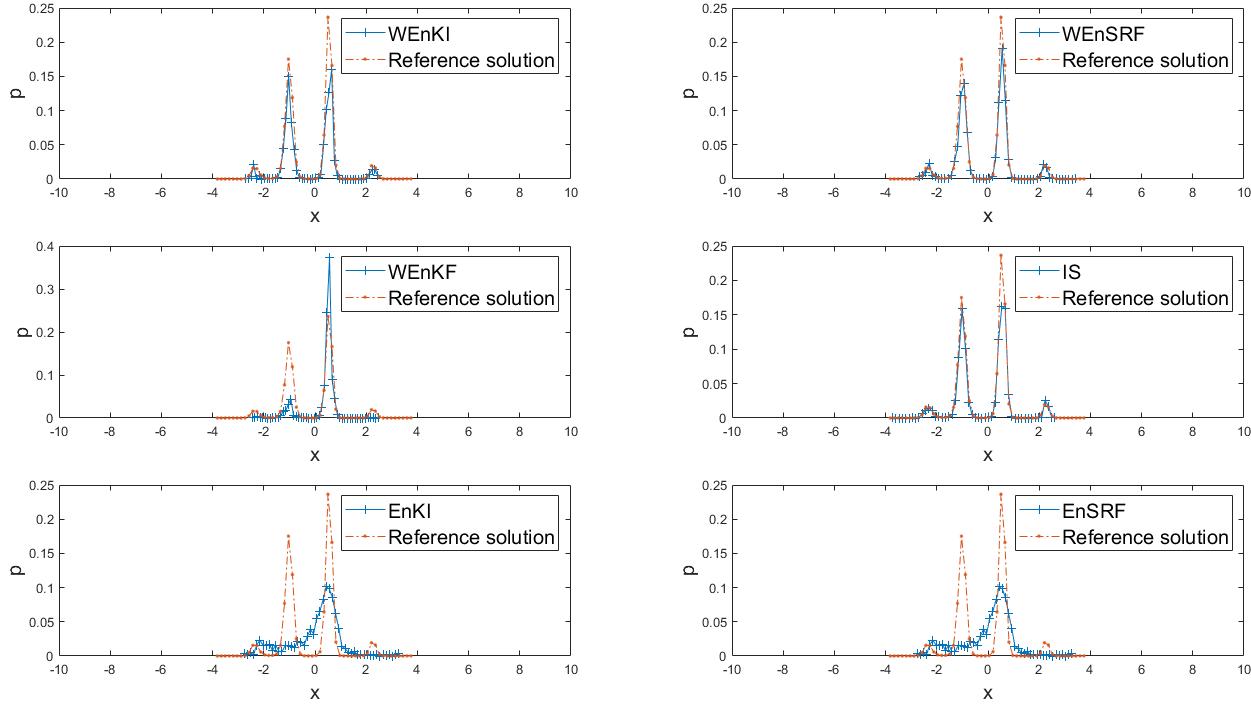}}
     \caption{Example $1$: from left top to bottom right: WEnKI; WEnSRF; WEnKF, as shown in Remark~\ref{rmk:WEnKF} and equation ~\eqref{weightupdatewenkf}; IS; EnKI and EnSRF. (All evolutional equation take $\Delta t = 10^{-3}$.) }
     \label{Figure1}
\end{figure}

\item Example 2: This is a highly nonlinear example with $4$-th power in $\mathcal{G}$: $G(u)=(u-3)^4-1$, and data is still set to be $y=0$. $N=2000$ and $\Delta t = 10^{-5}$. WEnKI and WEnSRF clearly outperform the others, see Figure~\ref{Figure3}. Note that in the experiment we find that for stability of the Euler solver~\eqref{eqn:wensrf_SDE}, and~\eqref{eqn:weki_SDE}, the time step is chosen to be rather small.
\begin{figure}[!ht]
     \centering
     \subfloat{\includegraphics[width=1\textwidth,height=0.4\textheight]{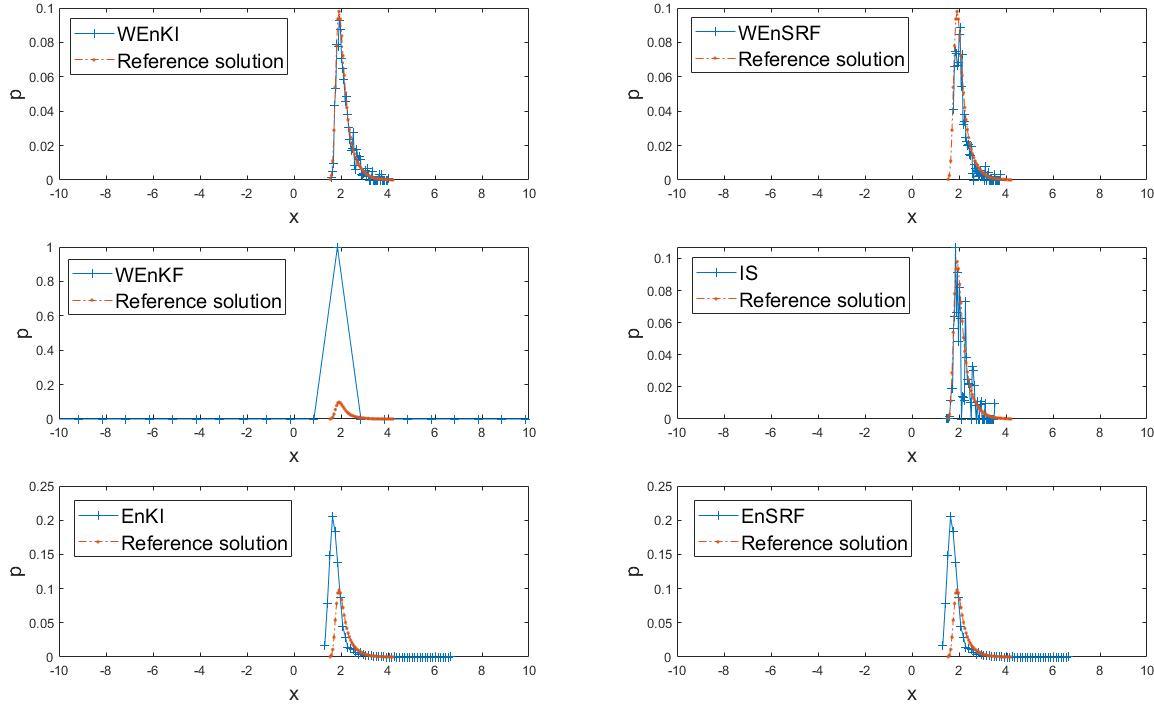}}
     \caption{Example $2$: from left top to bottom right: WEnKI; WEnSRF; WEnKF; IS; EnKI and EnSRF.}
     \label{Figure3}
\end{figure}

\item Example 3: In this example we set $\mathcal{G}(u)=(u-5)^2$ and data $y=0$. $N=2000$ and $\Delta t = 10^{-3}$. The posterior distribution has one peak, but is non-Gaussian. While the center of the prior is at $0$, the center of the likelihood function is at $u=5$: so there is a big shift of support from the prior to the posterior distribution. As seen in Figure \ref{Figure2}, both WEnKI and WEnSRF still capture the posterior distribution rather well. EnKI and EnSRF cannot capture the entire profile, but at least can move to fit relatively accurate support. WEnKF and IS completely fail.
\begin{figure}[!ht]
     \centering
     \subfloat{\includegraphics[width=1\textwidth,height=0.4\textheight]{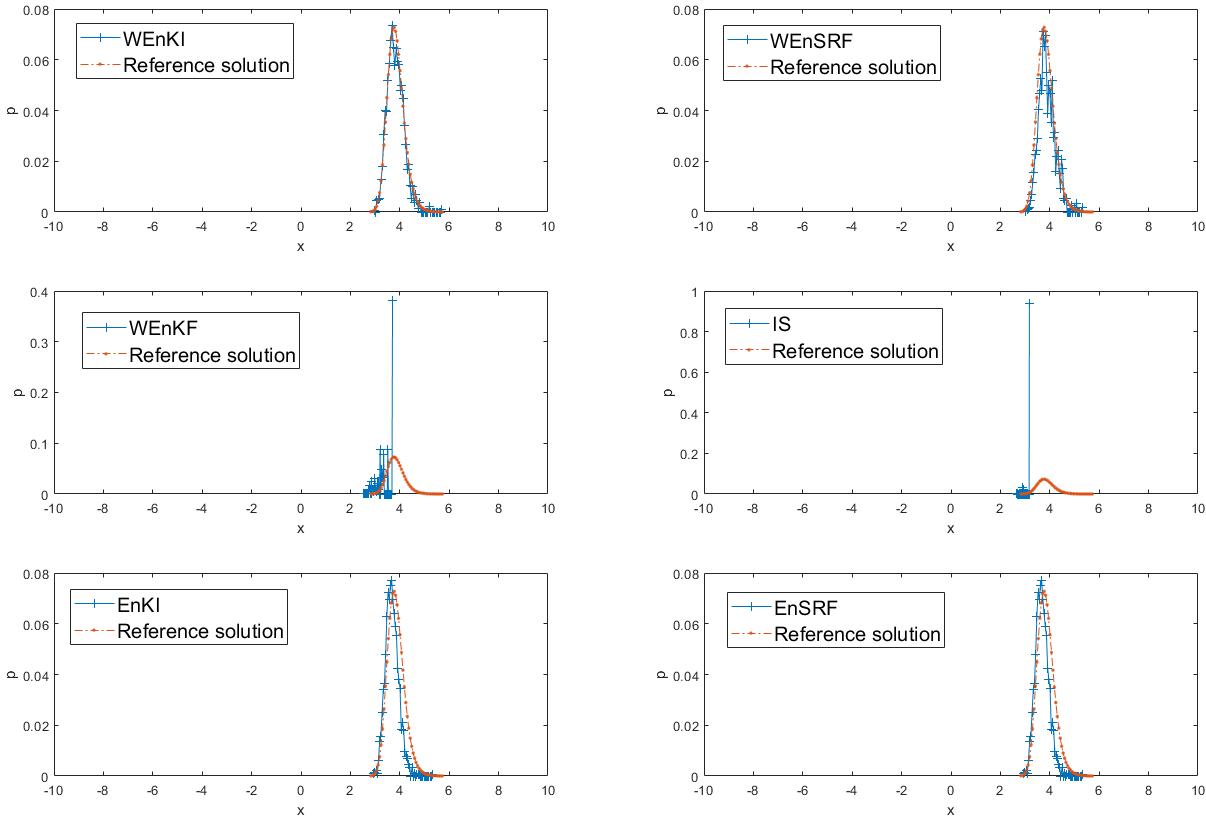}}
     \caption{Example $3$: from left top to bottom right: WEnKI; WEnSRF; WEnKF; IS; EnKI and EnSRF.}
     \label{Figure2}
\end{figure}
\end{enumerate}

To quantitatively study the behavior, we numerically compute the weights variance
\begin{equation}\label{var:weight}
\textrm{Var}(Nw(t))\approx\frac{1}{N}\sum^N_{n=1}|Nw^n(t)|^2-1
\end{equation}
of all three methods (WEnKI, WEnSRF, and IS). (For IS, the weight at $t$ is calculated  using $\rho(u,t)$ (defined in~\eqref{eqn:rho_t}).)  In Figure~\ref{Figure7}, we plot evolution of the weight variance with respect to $t$ in log scale (shifted by $1$ for positivity). It shows the variance of weights in IS quickly blows up in time, while the quantity for the other two keep reasonably bounded.
\begin{figure}[!ht]
     \centering
     \includegraphics[width=.65\textwidth]{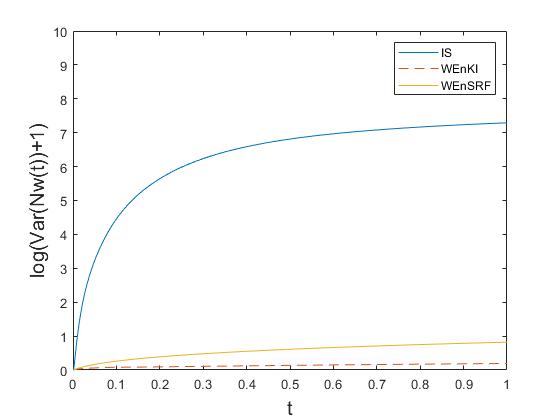}
     \caption{Example 3: $\log(\textrm{Var}(Nw(t))+1)$ for WEnKI, WEnSRF and IS.}
     \label{Figure7}
\end{figure}

As a demonstration of consistency, we compare moments computed using the four methods, and the reference solution, as shown in Table~\ref{table:example2}. It is clear that despite EnKI and EnSRF are visually close to the groundtruth solution, the errors in the moments are still rather large. This is expected: the groundtruth solution is not a Gaussian distribution, but the underlying assumption for EnKI and EnSRF to be valid is the Gaussianity.
\begin{table}[ht]
\caption{Error of moments estimation in Example 3}
\label{table:example2}
\begin{center}
\begin{tabular}{ |c|c|c|c|c|c|c|c|c|}
 \hline
&\multicolumn{2}{|c|}{WEnKI}&\multicolumn{2}{|c|}{WEnSRF}\\
\hline
Moments&Est.&Re. Error&Est.&Re. Error\\
\hline
$\EE|u|^1=3.84$&3.82&0.0056&3.88&0.0098\\
\hline
$\EE|u|^2=14.90$&14.73&0.0114&15.19&0.0192\\
\hline
$\EE|u|^3=58.22$&57.19&0.0177&59.86&0.0281\\
\hline
$\EE|u|^4=229.36$&223.79&0.0243&237.75&0.0366\\
\hline
$\EE|u|^5=911.22$&882.83&0.0312&951.95&0.0447\\
\hline
\end{tabular}
\begin{tabular}{ |c|c|c|c|c|c|c|c|c|}
 \hline
&\multicolumn{2}{|c|}{EnKI}&\multicolumn{2}{|c|}{EnSRF}\\
\hline
Moments&Est.&Re. Error&Est.&Re. Error\\
\hline
$\EE|u|^1=3.84$&3.69&0.0413&3.70&0.0391\\
\hline
$\EE|u|^2=14.90$&13.66&0.0833&13.73&0.0785 \\
\hline
$\EE|u|^3=58.22$&50.90&0.1258&51.35&0.1181 \\
\hline
$\EE|u|^4=229.36$&190.68&0.1687&193.24&0.1575 \\
\hline
$\EE|u|^5=911.22$&718.31&0.2117&732.17&0.1965 \\
\hline
\end{tabular}
\begin{tabular}{ |c|c|c|c|c|c|c|c|c|}
 \hline
&\multicolumn{2}{|c|}{WEnKF}&\multicolumn{2}{|c|}{IS}\\
\hline
Moments&Est.&Re. Error&Est.&Re. Error\\
\hline
$\EE|u|^1=3.84$&3.40&0.1156&3.52&0.0858 \\
\hline
$\EE|u|^2=14.90$&11.65&0.2181&12.37&0.1699 \\
\hline
$\EE|u|^3=58.22$&40.22&0.3093&43.57&0.2517 \\
\hline
$\EE|u|^4=229.36$&139.72&0.3908&153.56&0.3305 \\
\hline
$\EE|u|^5=911.22$&488.51&0.4639&541.71&0.4055 \\
\hline
\end{tabular}
\end{center}
\end{table}

\textcolor{black}{As noted before, the weight terms for the two methods are very complicated. The updating formula also requires the computation of the derivatives of $\mathcal{G}$. This will introduce a high cost in practice. We document the cost in Table~\ref{table:runningtime}. The weighted version of the algorithms almost double the cost. This is understandable. The number of ODEs that we need to compute is doubled: instead of computing $u_t$ only, we compute both $u_t$ and $w_t$.}
\begin{table}[ht]
\caption{Simulation time in Example 1-3}
\label{table:runningtime}
\begin{center}
\begin{tabular}{ |c|c|c|c|c|c|c|c|c|}
\hline
Case&WEnKI&WEnSRF&EnKI&EnSRF\\
\hline
Example 1 &0.362s&0.197s&0.138s&0.178s \\
\hline
Example 2&50.041s&41.739s&26.564s&18.518s \\
\hline
Example 3&0.198s&0.115s&0.120s&0.072s \\
\hline
\end{tabular}
\end{center}
\end{table}

\subsection{Two dimension example}
We also present some $2$-D examples. Normal distribution $\mathcal{N}(0,\mathrm{I}_2)$ is chosen as the prior distribution.
\begin{enumerate}[label=$\bullet$, wide]

\item Example 4: We consider likelihood function
\[
\exp(-\Phi(u;y))=\frac{1}{4}\sum^1_{i,j=0}\exp\left(-\frac{(u_1-a_{i,j})^2+(u_2-b_{i,j})^2}{0.2}\right)\,,
\]
with
\[
a=\begin{bmatrix}
6 & 3\\
3 & 0
\end{bmatrix},\quad
b=\begin{bmatrix}
3 & 6\\
0 & 3
\end{bmatrix}\,.
\]
This design of likelihood function induces two separate centers, as shown in Figure \ref{Figure4}. Here, we use $N=2000$ and choose $\Delta t=10^{-4}$ for WEnKI and WEnSRF. They capture the motion of the particles accurately. In comparison, IS loses a lot of particles. In this multimodal example, EnKI and EnSRF fail as expected due to the Gaussian assumption. \textcolor{black}{We should emphasize that WEnSRF and WEnKI are not perfect in practice. Indeed, they are nevertheless the corrected version of EnKI and EnSRF, the two methods that drove most of the particles to one (the red) block. The weighted correction makes sampling theoretically unbiased, but in the end more particles are left in this particular block.}
\begin{figure}[!ht]
     \centering
     \subfloat{\includegraphics[width=1\textwidth,height=0.4\textheight]{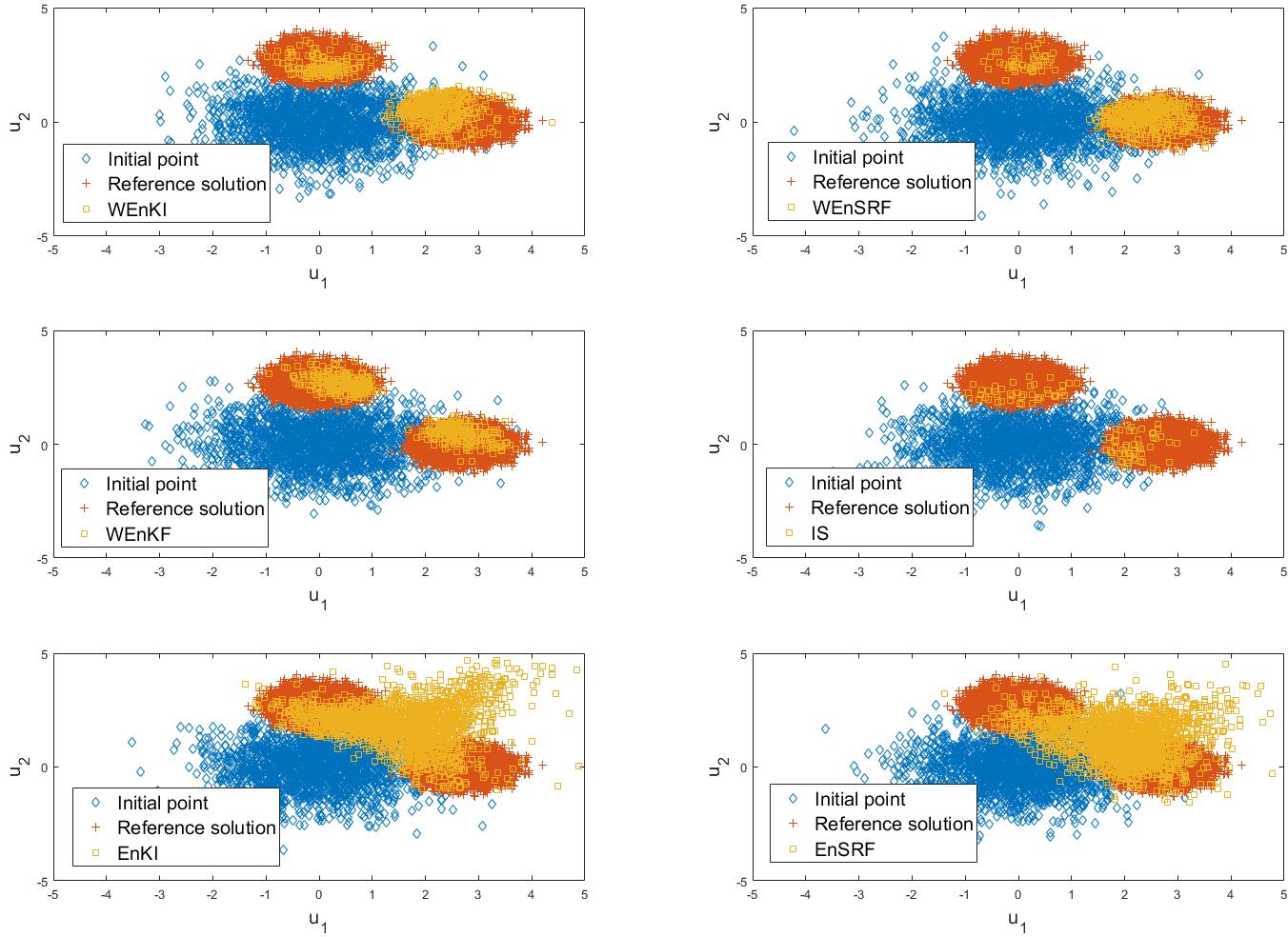}}
     \caption{Example $4$: from left top to bottom right: WEnKI; WEnSRF; WEnKF; IS; EnKI and EnSRF.}
     \label{Figure4}
\end{figure}

\item Example 5: In this case we consider $\mathcal{G}(u_1,u_2)=\left(g_1(u_1,u_2),g_2(u_1,u_2)\right)$ and $y=(0, 0)$, where
\begin{align*}
g_1(u_1,u_2)=(u_1-3)^2+\frac{(u_2-3)^2}{2},\quad g_2(u_1,u_2)=\frac{(u_1-3)^2}{2}+(u_2-3)^2\,.
\end{align*}
$N = 1000$ and $\Delta t=10^{-3}$ for WEnKI and WEnSRF. Results are presented in Figure \ref{Figure5}. Due to the form of $\mathcal{G}$, the center of the likelihood function is $(2,2)$ instead of $(0,0)$ for the prior distribution. Such transition of support is hard for IS to capture. After resampling, only a few samples survive. Visually EnKI and EnSRF still give satisfying results.
\begin{figure}[!ht]
     \centering
     \subfloat{\includegraphics[width=1\textwidth,height=0.4\textheight]{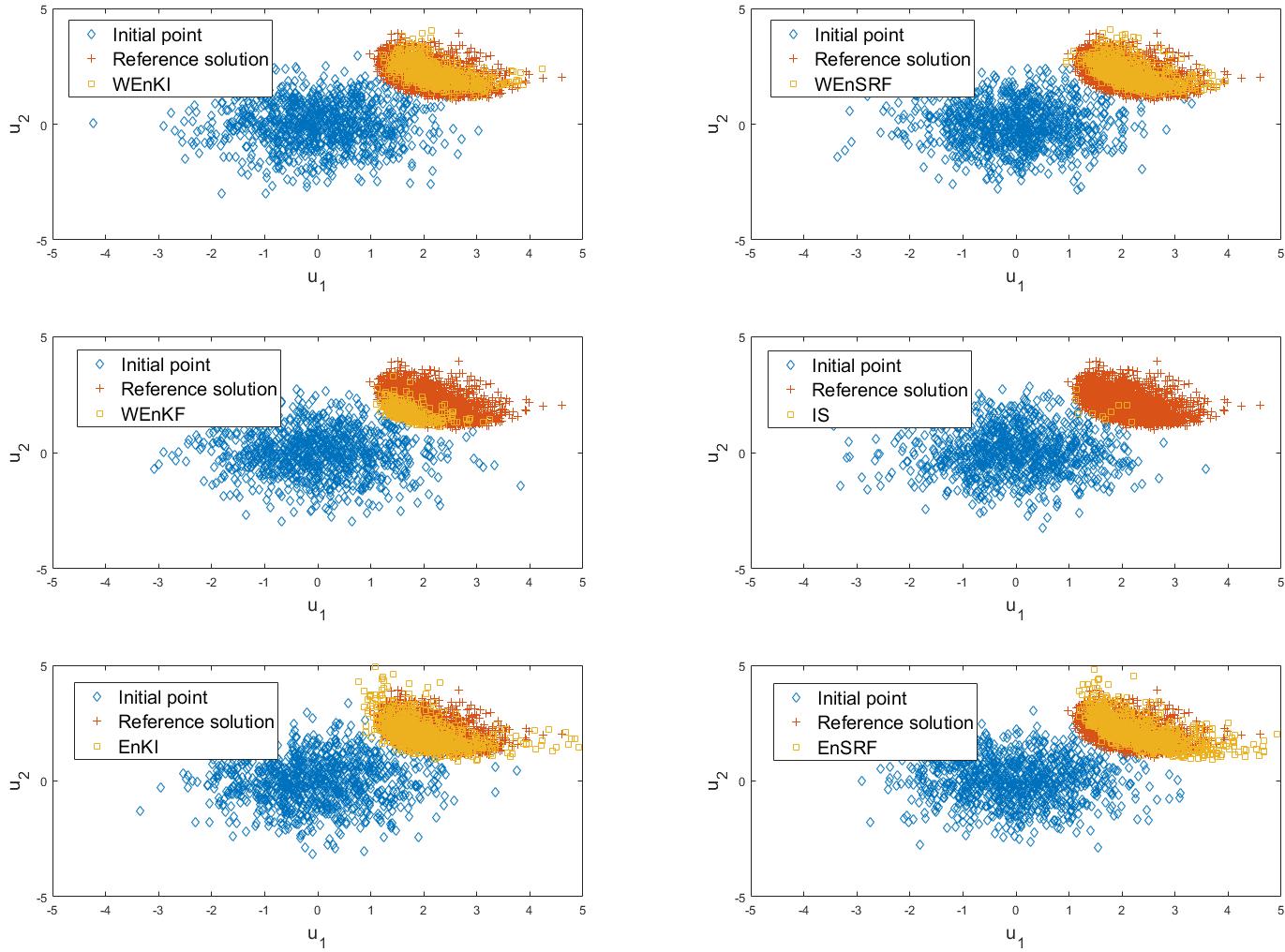}}
     \caption{Example $5$: from left top to bottom right: WEnKI; WEnSRF; WEnKF; IS; EnKI and EnSRF.}
     \label{Figure5}
\end{figure}

\end{enumerate}

To quantitatively understand the performance of the algorithms, we compute the variance of weight~\eqref{var:weight} and accuracy of moments estimation in this example. The variance of the weight, as a function of time, is plotted in Figure \ref{Figureweight:example4} in log scale (shifted by $1$ for positivity). As can be seen clearly, the weight of IS blows up quickly while the two newly proposed methods stay reasonable. In Table \ref{table:example4}, we tabulate the error of higher moments. Even though EnKI and EnSRF are visually good methods, in comparison, they do not capture the moments as well as their weighted versions. 

\begin{figure}[!ht]
     \centering
     \includegraphics[width = .65\textwidth]{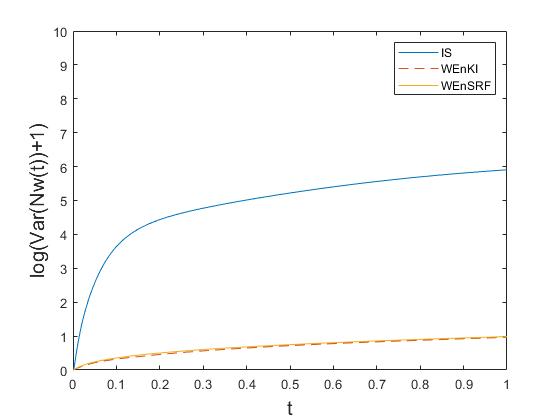}
     \caption{Example 5: $\log(\textrm{Var}(Nw(t))+1)$ for WEnKI, WEnSRF and IS}
     \label{Figureweight:example4}
\end{figure}
\begin{table}[ht]
\caption{Error of moments estimation in Example $5$}
\label{table:example4}
\begin{center}
\begin{tabular}{ |c|c|c|c|c|c|c|c|c|}
 \hline
&\multicolumn{2}{|c|}{WEnKI}&\multicolumn{2}{|c|}{WEnSRF}\\
\hline
Moments&Est.&Re. Error&Est.&Re. Error\\
\hline
$\EE|u|^1=3.32$&3.30&0.0055&3.32&0.0017\\
\hline
$\EE|u|^2=11.16$&10.99&0.0147&11.19&0.0023\\
\hline
$\EE|u|^3=38.05$&36.99&0.0279&38.12&0.0019\\
\hline
$\EE|u|^4=131.45$&125.53&0.0451&131.47&0.0001\\
\hline
$\EE|u|^5=460.56$&429.99&0.0664&459.16&0.0030\\
\hline
\end{tabular}
\begin{tabular}{ |c|c|c|c|c|c|c|c|c|}
 \hline
&\multicolumn{2}{|c|}{EnKI}&\multicolumn{2}{|c|}{EnSRF}\\
\hline
Moments&Est.&Re. Error&Est.&Re. Error\\
\hline
$\EE|u|^1=3.32$&2.96&0.1084&3.28&0.0112 \\
\hline
$\EE|u|^2=11.16$&9.07&0.1872&11.04&0.0111 \\
\hline
$\EE|u|^3=38.05$&29.17&0.2332&38.25&0.0053 \\
\hline
$\EE|u|^4=131.45$&100.32&0.2369&137.43&0.0455 \\
\hline
$\EE|u|^5=460.56$&379.73&0.1755&516.22&0.1208 \\
\hline
\end{tabular}
\begin{tabular}{ |c|c|c|c|c|c|c|c|c|}
 \hline
&\multicolumn{2}{|c|}{WEnKF}&\multicolumn{2}{|c|}{IS}\\
\hline
Moments&Est.&Re. Error&Est.&Re. Error\\
\hline
$\EE|u|^1=3.32$&3.40&0.1658&3.24&0.0245 \\
\hline
$\EE|u|^2=11.16$&7.72&0.3077&10.50&0.0592 \\
\hline
$\EE|u|^3=38.05$&21.74&0.4287&34.10&0.1037 \\
\hline
$\EE|u|^4=131.45$&61.62&0.5313&110.81&0.1571 \\
\hline
$\EE|u|^5=460.56$&175.99&0.6179&360.27&0.2178 \\
\hline
\end{tabular}
\end{center}
\end{table}

\section{Conclusion}
We conclude the paper with a few remarks of the proposed algorithms.

Since EnKI was proposed in~\cite{Iglesias_2013}, the mystery of if and how it works for the nonlinear case has attracted a lot of attention. The surrounding work, such as the wellposedness of the coupled SDE~\cite{Schilling2018}, the wellposedness of the PDE~\cite{ding2019ensemble, ding2019ensembleS}, the mean-field limit of SDE to the PDE, and the convergence rate~\cite{ding2019ensemble, Fournier2015}, and convergence as an optimization method~\cite{chada2019tikhonov,chada2019convergence}, have all been studied in depth, and the use of similar idea leads to development of new algorithms~\cite{lu2019accelerating,ding2019ensembleS}. The investigation into the core sampling problem with nonlinear forward map, however, is thin.

\textcolor{black}{In this paper, by adding the weights to the particles, we are able to correct EnKI (and similarly EnSRF) to ensure the consistency of the algorithm for nonlinear $\mathcal{G}$. The derivation, though tedious, is mathematically straightforward. The resulting ``weight" factor~\eqref{R1}, however, is mathematically messy and physically not intuitive at all. We would like to emphasize that this nonphysical weight term is uniquely determined once the flow is set, namely, if one follows the flow of EnKI,
\[
\rd u^n=\mathrm{Cov}_{up}\Gamma^{-1}\left(y-\MCG(u^n_t)\right)dt+\mathrm{Cov}_{up}\Gamma^{-1/2}\rd W_t\,,
\]
so that in time, the flow provides a linear interpolation between the origin and the target on the logarithmic scale:
\[
\rho(u,t) \sim \mu_\pri\exp\{-t|y-\MCG(u)|^2_{\Gamma}/2\}\,,
\]
then there will be no other ways to define the weight term, and it has to be as tedious and nonphysical as was derived in this paper. This naturally leads to the question if some modifications to the flow can result a physically more meaningful weight function. This, however, is beyond the scope of the current paper.}
\newpage
\appendix
\textcolor{black}{
\section{Derivatives of $\rho$}\label{Derivetiveofrho}
Recall the definition of $\rho$, we compute its time derivative to have:
\[
\begin{aligned}
\partial_t\rho&=-\frac{Z'(t)}{Z^2(t)}\exp\{-t\Phi(u;y)\}\rho_\pri(u)-\frac{\Phi(u;y)}{Z(t)}\exp\{-t\Phi(u;y)\}\rho_\pri(u)\\
& = -\frac{Z'(t)}{Z(t)}\rho -\Phi(u;y)\rho\,.
\end{aligned}
\]
Since $\Phi(u;y)=\frac{1}{2}|y-\MCG(u)|^2_{\Gamma}$
\[
-\frac{Z'(t)}{Z(t)}=-\int -\Phi(u;y) \rho(u,t)du = \EE_{\rho(t)}\left(\frac{1}{2}|y-\mathcal{G}(u)|^2_{\Gamma}\right)\,,
\]
we obtain~\eqref{partialtmu1}.
\\
Compute the $u$ derivative, we have:
\begin{equation}\label{derivative_u}
\nabla \rho=\frac{1}{Z(t)}(\nabla \exp\{-t\Phi(u;y)\})\rho_\pri(u)+\frac{1}{Z(t)}\exp\{-t\Phi(u;y)\}\nabla \rho_\pri(u)\,.
\end{equation}
Noticing
\begin{equation*}
\nabla \exp\{-t\Phi(u;y)\}=-t\nabla \Phi(u;y)\exp\{-t\Phi(u;y)\}
\end{equation*}
and
\begin{equation*}
\nabla \rho_\pri(u)=-\Gamma^{-1}_0(u-u_0)\rho_\pri(u)\,,
\end{equation*}
we plug them in~\eqref{derivative_u} to obtain~\eqref{partialxmu} for $\nabla \rho = \mathcal{V}\rho$.
\\
To compute the hessian, we repeat the process above for
\[
\mathcal{H}_u\rho = \nabla\mathcal{V}u + \mathcal{V}\left(\nabla\rho\right)^\top\,.
\]
While $\mathcal{V}\left(\nabla\rho\right)^\top$ contributes $\mathcal{V}\mathcal{V}^\top\rho$, the derivative of $\mathcal{V}$ become $-t\left(\nabla \mathcal{G}\right)^\top\Gamma^{-1}\nabla\mathcal{G} +t\mathcal{W}$, leading to~\eqref{partialxxmu} in the end.
}
\section{Proof of Proposition~\ref{thm:newweightbound}}\label{Appendix1}
In this appendix, we derive the explicit bound for $\mathcal{R}_i$ in the following lemma, and show Proof of Proposition~\ref{thm:newweightbound}. It plays the crucial role in Theorem~\ref{thm:betterestimation}.
\begin{lemma}\label{lem:boundforR1R3}
Under Assumption \ref{Gassum2}, for all $0<t<1$, $\mathcal{R}_1$, $\mathcal{R}_2$ and $\mathcal{R}_3$ defined in~\eqref{R1} satisfy:
\begin{itemize}
    \item For $\mathcal{R}_1$
\begin{equation}\label{boundr1:new}
\left|\mathcal{R}_1(u,t)\right|\leq C_2(\mathrm{Var}^{\rho(t)}(u))^2+C_1\mathrm{Var}^{\rho(t)}(u)\,.
\end{equation}
\item For $\mathcal{R}_2$
\begin{equation}\label{boundr2:new}
\begin{aligned}
2|\mathcal{R}_2(u^n_t,t)| \leq & \|\Gamma^{-1}\|_2(\Lambda^2|u^\ast_{\mathsf{A}}-\overline{u}^{\rho(t)}|^2+|r|^2+M^2)\,.
\end{aligned}
\end{equation}
\item For $\mathcal{R}_3$
\begin{equation}\label{boundr3:new}
\left|\mathcal{R}_3(u,t)\right|\leq C_3(\mathrm{Var}^{\rho(t)}(u))^2\,.
\end{equation} 
\item More carefully:
\begin{equation}\label{R2bestestimation}
\begin{aligned}
|\mathcal{R}_2(u^n_t,t)|\leq &C_4\left|\overline{u}^{\rho(t)}-u^\ast_{\mathsf{A}}\right|\left[\left|\overline{u}^{\rho(t)}-\Cov^{\rho(t)}_{u,u}(\Cov_\mathsf{A}(t))^{-1}u_\mathsf{A}\right|+\|\Cov_{u,u}^{\rho(t)}\|_2\Lambda_1\right]\\
&+C_4\|\Cov^{\rho(t)}_{\Rm,u}\|_2(|u_\mathsf{A}|+\Lambda_1)\\
&+C_4\left(\left|\overline{u}^{\rho(t)}-u^\ast_{\mathsf{A}}\right|\left\|I-(\Cov_\mathsf{A})^{-1}\Cov^{\rho(t)}_{u,u}\right\|_2+\|\Cov^{\rho(t)}_{\Rm,u}\|_2\right)|u^n_t|\,.
\end{aligned}
\end{equation}
\end{itemize}
In the equation~$\mathrm{Var}^{\rho(t)}(u)=\mathrm{Tr}\left(\Cov^{\rho(t)}_{u,u}\right)$ and all constants $C_1,\cdots,C_4$ are constants depending on $\Lambda$, $\|\Gamma^{-1}\|_2$, $\|\Gamma^{-1}_0\|_2$, $|\mathsf{r}|$, and $M$.
\end{lemma}

\begin{proof}
This comes from direct calculation. Firstly to show~\eqref{boundr1:new}, we plug \eqref{linear} into \eqref{R1}:
\begin{equation*}
\begin{aligned}
&\mathcal{R}_1(u,t)\\
=&\frac{1}{2}\mathrm{Tr}\left\{\Cov_{uu}^{\rho(t)}\mathsf{A}^\top\Gamma^{-1}\mathsf{A}-2\mathsf{A}^\top\Gamma^{-1}\mathsf{A}\Cov_{uu}^{\rho(t)}\right\}\\
\quad&+\frac{1}{2}\mathrm{Tr}\left\{\Cov_{uu}^{\rho(t)}\mathsf{A}^\top\Gamma^{-1}\mathsf{A}\Cov_{uu}^{\rho(t)}\left[t\left(\nabla \MCG(u)\right)^\top\Gamma^{-1}\nabla \MCG(u)+\Gamma^{-1}_0\right]\right\}\\
\quad&+\frac{1}{2}\mathrm{Tr}\left\{\Cov_{\mathsf{m}\mathsf{m}}^{\rho(t)}\Gamma^{-1}-2\left(\nabla \mathsf{m}\right)^\top\Gamma^{-1}\Cov_{\mathsf{m}u}^{\rho(t)}\right\}\\
\quad&+\frac{1}{2}\mathrm{Tr}\left\{\Cov_{u\mathsf{m}}^{\rho(t)}\Gamma^{-1}\Cov_{\mathsf{m}u}^{\rho(t)}\left[t\left(\nabla \MCG(u)\right)^\top\Gamma^{-1}\nabla \MCG(u)+\Gamma^{-1}_0\right]\right\}\,,
\end{aligned}
\end{equation*}
where we use \eqref{eqn:assum2} and the first term is less than $0$. Notice
\[
\begin{aligned}
\mathrm{Var}^{\rho(t)}(\sm(u))&=\mathrm{Tr}\left(\Cov^{\rho(t)}_{\Rm,\Rm}\right)=\EE^{\rho(t)}\left|\Rm(u)-\overline{\Rm}\right|^2\\
&\leq \EE^{\rho(t)}\left|\Rm(u)-\Rm(\overline{u})\right|^2\leq \Lambda^2\EE^{\rho(t)}\left|u-\overline{u}\right|^2=\Lambda^2\mathrm{Var}^{\rho(t)}(u)\,,
\end{aligned}
\]
we have the following five inequalities:
\begin{equation}
\begin{aligned}
\mathrm{Tr}\left\{\Cov_{uu}^{\rho(t)}\mathsf{A}^\top\Gamma^{-1}\mathsf{A}\right\}&\leq \Lambda^2\|\Gamma^{-1}\|_2\mathrm{Var}^{\rho(t)}(u)\,,\\
\mathrm{Tr}\left\{\Cov_{\mathsf{m}\mathsf{m}}^{\rho(t)}\Gamma^{-1}\right\}&\leq \|\Gamma^{-1}\|_2\mathrm{Var}^{\rho(t)}(\sm(u))\leq \Lambda^2\|\Gamma^{-1}\|_2\mathrm{Var}^{\rho(t)}(u)\,,\\
\left|\mathrm{Tr}\left\{\Gamma^{-1}\Cov_{\mathsf{m}u}^{\rho(t)}\right\}\right|&=\EE \left(\Rm(u)-\overline{\Rm}\right)^\top\Gamma^{-1} (u-\overline{u})\leq 
\|\Gamma^{-1}\|_2\EE\left|\Rm(u)-\overline{\Rm}\right|\left|u-\overline{u}\right|\\
&\leq\|\Gamma^{-1}\|_2(\mathrm{Var}^{\rho(t)}(\sm(u)))^{1/2}(\mathrm{Var}^{\rho(t)}(u))^{1/2}\\
&\leq \Lambda\|\Gamma^{-1}\|_2\mathrm{Var}^{\rho(t)}(u)\,,
\end{aligned}
\end{equation}
and furthermore
\begin{equation}\label{Covestimation1}
\begin{aligned}
\mathrm{Tr}\left\{\Cov_{uu}^{\rho(t)}\mathsf{A}^\top\Gamma^{-1}\mathsf{A}\Cov_{uu}^{\rho(t)}\right\}&\leq \mathrm{Tr}\left\{\Cov_{uu}^{\rho(t)}\mathsf{A}^\top\Gamma^{-1}\mathsf{A}\right\}\|\Cov_{uu}^{\rho(t)}\|_2\\
&\leq \Lambda^2\|\Gamma^{-1}\|_2(\mathrm{Var}^{\rho(t)}(u))^2\,,
\end{aligned}
\end{equation}
and
\begin{equation}\label{Covestimation2}
\begin{aligned}
\mathrm{Tr}\left\{\Cov_{u\sm}^{\rho(t)}\Gamma^{-1}\Cov_{\sm u}^{\rho(t)}\right\}
&\leq \|\Gamma^{-1}\|_2\|\Cov_{u\sm}^{\rho(t)}\|^2_F\leq \|\Gamma^{-1}\|_2\EE|u-\overline{u}|^2\EE|\Rm-\overline{\Rm}|^2\\
&\leq \|\Gamma^{-1}\|_2(\mathrm{Var}^{\rho(t)}(\sm(u)))(\mathrm{Var}^{\rho(t)}(u))\\
&\leq \Lambda^2\|\Gamma^{-1}\|_2(\mathrm{Var}^{\rho(t)}(u))^2\,,
\end{aligned}
\end{equation}
we have
\[
|\mathcal{R}_1(u,t)|\leq \frac{3}{2}\Lambda^2\|\Gamma^{-1}\|_2\mathrm{Var}^{\rho(t)}(u)+\left[t\Lambda^2\|\Gamma^{-1}\|_2+\|\Gamma^{-1}_0\|_2\right]\Lambda^2\|\Gamma^{-1}\|_2(\mathrm{Var}^{\rho(t)}(u))^2\,.
\]
Let
\[
C_1=\frac{3}{2}\Lambda^2\|\Gamma^{-1}\|_2\,,\quad C_2=\left[t\Lambda^2\|\Gamma^{-1}\|_2+\|\Gamma^{-1}_0\|_2\right]\Lambda^2\|\Gamma^{-1}\|_2\,,
\]
we obtain \eqref{boundr1:new}.  To bound $\mathcal{R}_3$, we first notice 
\[
\mathcal{W}(u)=\left[(\partial_1\nabla\Rm(u))^\top\Gamma^{-1}(\mathsf{r}-\Rm(u))\,, \cdots\,, (\partial_L\nabla\Rm(u))^\top\Gamma^{-1}(\mathsf{r}-\Rm(u))\right]\,.
\]
by plugging in \eqref{linear},\eqref{eqn:errorsfr},\eqref{r:orthorgonal}. Then we have
\[
\begin{aligned}
\left|\mathcal{R}_3(u,t)\right|&\leq \frac{t\mathrm{Tr}\left\{\Cov^{\rho(t)}_{uu}\mathsf{A}^\top\Gamma^{-1}\mathsf{A}\Cov^{\rho(t)}_{uu}+\Cov^{\rho(t)}_{u\sm}\Gamma^{-1}\Cov^{\rho(t)}_{\sm u}\right\}\|\mathcal{W}(u)\|_2}{2}\\
&\leq C'(\Lambda,\|\Gamma^{-1}\|_2)(1+|\mathsf{r}|)\Lambda^2\|\Gamma^{-1}\|_2(\mathrm{Var}^{\rho(t)}(u))^2\,,
\end{aligned}
\]
where we use \eqref{Covestimation1},\eqref{Covestimation2} and
\begin{align*}
\|\mathcal{W}(u)\|_2\leq &\|\mathcal{W}(u)\|_F\leq\frac{L}{2}\left(\|\mathcal{H}\left(|\Rm|^2_{\Gamma}\right)\|_2+\|\Gamma^{-1}\|_2\|\nabla \Rm\|^2_2\right)\\
&+|\mathsf{r}|\|\Gamma^{-1}\|_2\max_{1\leq i\leq L}\{\|\partial_i\nabla \Rm\|_2\}\\
\leq &C'(\Lambda,\|\Gamma^{-1}\|_2)(1+|\mathsf{r}|)\,.
\end{align*}
We obtain \eqref{boundr3:new} by defining $C_3=C'(\Lambda,\|\Gamma^{-1}\|_2)(1+|\mathsf{r}|)\Lambda^2\|\Gamma^{-1}\|_2$.

To show~\eqref{boundr2:new}, we simply plug in the weak nonlinearity assumption for:
\begin{equation*}
\begin{aligned}
2|\mathcal{R}_2(u^n_t,t)|\leq &\left(y-\overline{\MCG}^{\rho(t)}\right)^\top\Gamma^{-1}\left(y-\overline{\MCG}^{\rho(t)}\right)= \left|\mathsf{A}\left(u^\ast_{\mathsf{A}}-\overline{u}^{\rho(t)}\right)\right|^2_{\Gamma}+\left|\mathsf{r}-\overline{\sm}^{\rho(t)}\right|^2_{\Gamma}\\
\leq & \|\Gamma^{-1}\|_2(\Lambda^2|u^\ast_{\mathsf{A}}-\overline{u}^{\rho(t)}|^2+|r|^2+M^2)\,.
\end{aligned}
\end{equation*}

Finally,
\begin{equation}\label{R2Tfirst}
\begin{aligned}
\mathcal{R}_2(u^n_t,t)&=\frac{1}{2}\left|\mathsf{A}(\overline{u}^{\rho(t)}-u^\ast_{\mathsf{A}})\right|^2_\Gamma-\frac{1}{2}\left|\mathsf{A}(u^\ast_{\mathsf{A}}-u^n_t)-\mathsf{A}\Cov^{\rho(t)}_{u,u}\mathcal{V}(u^n_t,t)\right|^2_\Gamma\\
&+\frac{1}{2}\left|\mathsf{r}-\Rm(u^n_t)\right|^2_\Gamma-\frac{1}{2}\left|\mathsf{r}-\Rm(u^n_t)-\Cov^{\rho(t)}_{\Rm,u}\mathcal{V}(u^n_t,t)\right|^2_\Gamma\,.
\end{aligned}
\end{equation}

According to the definition of $\mathcal{V}$,
\begin{equation}\label{eqn:Vnt}
\begin{aligned}
\mathcal{V}(u^n_t,t)&=t\mathsf{A}^\top\Gamma^{-1}\mathsf{A}(u^\ast_{\mathsf{A}}-u^n_t)-\Gamma^{-1}_0(u^n_t-u_0)+t\left(\nabla \Rm\right)^\top\Gamma^{-1}(\mathsf{r}-\Rm(u^n_t))\\
&=(\Cov_\mathsf{A}(t))^{-1}(u_\mathsf{A}-u^n_t)+t\left(\nabla \Rm\right)^\top\Gamma^{-1}(\mathsf{r}-\Rm(u^n_t))\,,
\end{aligned}
\end{equation}
and thus
\[
\begin{aligned}
&\frac{1}{2}\left|\mathsf{A}(u^\ast_{\mathsf{A}}-u^n_t)-\mathsf{A}\Cov^{\rho(t)}_{u,u}\mathcal{V}(u^n_t,t)\right|^2_\Gamma\\
=&\frac{1}{2}\left|\mathsf{A}\left[(I-\Cov^{\rho(t)}_{u,u}(\Cov_\mathsf{A}(t))^{-1})u^n_t+\Cov^{\rho(t)}_{u,u}(\Cov_\mathsf{A}(t))^{-1}u_\mathsf{A}(t)-u^\ast_{\mathsf{A}}+\mathcal{M}(u)\right]\right|^2_\Gamma\,,
\end{aligned}
\]
where
\[
\mathcal{M}(u)=t\Cov^{\rho(t)}_{u,u}\left(\nabla \Rm\right)^\top\Gamma^{-1}(\mathsf{r}-\Rm(u))\,.
\]
This means the first two terms in~\eqref{R2Tfirst} are controlled by:
\begin{equation}\label{R2firsttwoterms}
\begin{aligned}
&\frac{1}{2}\left|\mathsf{A}(\overline{u}^{\rho(t)}-u^\ast_{\mathsf{A}})\right|^2_\Gamma-\frac{1}{2}\left|\mathsf{A}(u^\ast_{\mathsf{A}}-u^n_t)-\mathsf{A}\Cov^{\rho(t)}_{u,u}\mathcal{V}(u^n_t,t)\right|^2_\Gamma\\
\leq&\frac{1}{2}\left|\mathsf{A}(\overline{u}^{\rho(t)}-u^\ast_{\mathsf{A}})\right|^2_\Gamma\\
&-\frac{1}{2}\left|\mathsf{A}\left[(I-\Cov^{\rho(t)}_{u,u}(\Cov_\mathsf{A})^{-1})u^n_t+\Cov^{\rho(t)}_{u,u}(\Cov_\mathsf{A}(t))^{-1}u_\mathsf{A}-u^\ast_{\mathsf{A}}+\mathcal{M}(u)\right]\right|^2_\Gamma\\
\leq&C'\left|\mathsf{A}(\overline{u}^{\rho(t)}-u^\ast_{\mathsf{A}})\right|_\Gamma\\
&\cdot\left|\mathsf{A}\left[(I-(\Cov_\mathsf{A})^{-1}\Cov^{\rho(t)}_{u,u})u^n_t+\Cov^{\rho(t)}_{u,u}(\Cov_\mathsf{A}(t))^{-1}u_\mathsf{A}-\overline{u}^{\rho(t)}+\mathcal{M}(u)\right]\right|_\Gamma\\
\leq &C'\left|\overline{u}^{\rho(t)}-u^\ast_{\mathsf{A}}\right|\left[\left|\overline{u}^{\rho(t)}-\Cov^{\rho(t)}_{u,u}(\Cov_\mathsf{A}(t))^{-1}u_\mathsf{A}\right|+\|\Cov_{u,u}^{\rho(t)}\|_2\Lambda_1\right]\\
&+C'\left|\overline{u}^{\rho(t)}-u^\ast_{\mathsf{A}}\right|\left\|I-(\Cov_\mathsf{A})^{-1}\Cov^{\rho(t)}_{u,u}\right\|_2|u^n_t|\,,
\end{aligned}
\end{equation}
where $C'$ is a constant depending on $\Lambda$, $\|\Gamma^{-1}\|_2$, $\|\Gamma^{-1}_0\|_2$, $|\mathsf{r}|$ and $M$. Furthermore, the latter two terms in~\eqref{R2Tfirst} are bounded by:
\begin{equation}\label{R2secondtwoterms}
\begin{aligned}
&\frac{1}{2}\left|\mathsf{r}-\Rm(u^n_t)\right|^2_\Gamma-\frac{1}{2}\left|\mathsf{r}-\Rm(u^n_t)-\Cov^{\rho(t)}_{\Rm,u}\mathcal{V}(u^n_t,t)\right|^2_\Gamma\\
\leq &\left\langle \Gamma^{-1}(\mathsf{r}-\Rm(u^n_t)),\Cov^{\rho(t)}_{\Rm,u}(\Cov_\mathsf{A}(t))^{-1}(u_\mathsf{A}-u^n_t)\right\rangle\\
&+\left\langle \Gamma^{-1}(\mathsf{r}-\Rm(u^n_t)),t\Cov^{\rho(t)}_{\Rm,u}\left(\nabla \Rm\right)^\top\Gamma^{-1}(\mathsf{r}-\Rm(u^n_t))\right\rangle\\
\leq &C''\|\Cov^{\rho(t)}_{\Rm,u}\|_2(|u_\mathsf{A}|+|u^n_t|)+C''\Lambda_1\|\Cov^{\rho(t)}_{\Rm,u}\|_2\\
\leq &C''\|\Cov^{\rho(t)}_{\Rm,u}\|_2(|u_\mathsf{A}|+\Lambda_1)+C''\|\Cov^{\rho(t)}_{\Rm,u}\|_2|u^n_t|\,,
\end{aligned}
\end{equation}
where $C''$ is a constant depending on $\Lambda$, $\|\Gamma^{-1}\|_2$, $\|\Gamma^{-1}_0\|_2$, $|\mathsf{r}|$, and $M$.

These together give the upper bound for $\mathcal{R}_2$. Call the constant $C_4$, we finish the proof.
\end{proof}

Now we are ready to prove Proposition~\ref{thm:newweightbound}.
\begin{proof}
We first estimate $\EE\left(|u^n_t|^2(\omega^n_t)^2\right)$. Using \eqref{eqn:weki_SDE} and It\^o's formula, we obtain
\[
\rd |u^n_t|^2(\omega^n_t)^2= 2(\omega^n_t)^2 \left\langle \rd u^n_t,u^n_t\right\rangle+(\omega^n_t)^2\left\langle \rd u^n_t,\rd u^n_t\right\rangle+2|u^n_t|^2 \omega^n_t\rd \omega^n_t\,,
\]
and thus
\begin{equation}\label{eqn:termbound}
\begin{aligned}
\frac{\rd}{\rd t} \EE|u^n_t|^2(\omega^n_t)^2 = &2\EE\left((\omega^n_t)^2\left\langle \mathrm{Cov}_{up}^{\rho(t)}\Gamma^{-1}\left(y-\MCG(u^n_t)\right),u^n_t\right\rangle\right)\\
&+\EE\left\{(\omega^n_t)^2\mathrm{Tr}\left(\mathrm{Cov}_{up}^{\rho(t)}\Gamma^{-1}\mathrm{Cov}_{pu}^{\rho(t)}\right)\right\}\\
&+2\EE\left\{|u^n_t|^2 (\omega^n_t)^2\left|\mathcal{R}_1(u^n_t,t)+\mathcal{R}_2(u^n_t,t)+\mathcal{R}_3(u^n_t,t)\right|\right\}\,.
\end{aligned}
\end{equation}

We now bound the two terms respectively. For the first:
\begin{equation}\label{boundforfirstterm:new}
\begin{aligned}
&\EE\left\{(\omega^n_t)^2\left[\left\langle \mathrm{Cov}_{up}^{\rho(t)}\Gamma^{-1}\left(y-\MCG(u^n_t)\right),u^n_t\right\rangle+\frac{1}{2}\mathrm{Tr}\left(\mathrm{Cov}_{up}^{\rho(t)}\Gamma^{-1}\mathrm{Cov}_{pu}^{\rho(t)}\right)\right]\right\}\\
=&\EE\left\{(\omega^n_t)^2\left[\left\langle \mathrm{Cov}_{uu}^{\rho(t)}\mathsf{A}^\top\Gamma^{-1}\mathsf{A}\left(u^\ast_{\mathsf{A}}-u^n_t\right)+\mathrm{Cov}_{u\mathsf{m}}^{\rho(t)}\Gamma^{-1}(\mathsf{r}-\Rm(u^n_t)),u^n_t\right\rangle\right]\right\}\\
&+\frac{1}{2}\EE\left\{(\omega^n_t)^2\left[\mathrm{Tr}\left(\mathrm{Cov}_{uu}^{\rho(t)}\mathsf{A}^\top\Gamma^{-1}\mathsf{A}\mathrm{Cov}_{u u}^{\rho(t)}+\mathrm{Cov}_{u\sm}^{\rho(t)}\Gamma^{-1}\mathrm{Cov}_{\sm u}^{\rho(t)}\right)\right]\right\}\\
\leq& \Lambda^2\|\Gamma^{-1}\|_2\mathrm{Var}^{\rho(t)}(u)\EE\left[(|u^\ast_{\mathsf{A}}||u^n_t|+|u^n_t|^2+|\mathsf{r}||u^n_t|+M|u^n_t|)(\omega^n_t)^2\right]\\
&+\Lambda^2\|\Gamma^{-1}\|_2(\mathrm{Var}^{\rho(t)}(u))^2\EE(\omega^n_t)^2\\
\leq& \Lambda^2\|\Gamma^{-1}\|_2\mathrm{Var}^{\rho(t)}(u)\left\{\left(\frac{|u^\ast_{\mathsf{A}}|+|\mathsf{r}|+M}{2}+1\right)\EE\left[|u^n_t|^2(\omega^n_t)^2\right]\right.\\
&\left.+\left[(\mathrm{Var}^{\rho(t)}(u))+\left(\frac{|u^\ast_{\mathsf{A}}|+|\mathsf{r}|+M}{2}\right)\right]\EE(\omega^n_t)^2\right\}\,,
\end{aligned}
\end{equation}
where we use \eqref{linear} in the first equality and that
\[
\begin{aligned}
\left\langle \mathrm{Cov}_{uu}^\rho\mathsf{A}^\top\Gamma^{-1}\mathsf{A}\left(u^\ast_{\mathsf{A}}-u^n_t\right),u^n_t\right\rangle&\leq \|\mathrm{Cov}_{uu}^\rho\mathsf{A}^\top\Gamma^{-1}\mathsf{A}\|_2\left[(|u^\ast_{\mathsf{A}}||u^n_t|+|u^n_t|^2\right]\\
&\leq \Lambda^2\|\Gamma^{-1}\|_2\mathrm{Var}^{\rho(t)}(u)\left[|u^\ast_{\mathsf{A}}||u^n_t|+|u^n_t|^2\right]\,
\end{aligned}
\]
and
\[
\begin{aligned}
\left\langle \mathrm{Cov}_{u\Rm}^\rho\Gamma^{-1}(\mathsf{r}-\Rm(u)),u^n_t\right\rangle&\leq \left[\EE \left|\Rm(u)-\overline{\Rm}\right|\|\Gamma^{-1}\|_2 \left|u-\overline{u}\right|\right]\left[|\mathsf{r}||u^n_t|+M|u^n_t|\right]\\
&\leq \Lambda^2\|\Gamma^{-1}\|_2\mathrm{Var}^{\rho(t)}(u)\left[|\mathsf{r}||u^n_t|+M|u^n_t|\right]\,,
\end{aligned}
\]
where~\eqref{Covestimation1} and~\eqref{Covestimation2} are applied.

For second term in \eqref{eqn:termbound}, we simply apply the inequalities~\eqref{boundr1:new}-\eqref{boundr2:new}. These together provides the estimate of~\eqref{eqn:termbound} as
\begin{equation}\label{changeofu2w2}
\begin{aligned}
\frac{\rd}{\rd t} \EE|u^n_t|^2(\omega^n_t)^2 \leq &\widetilde{C}\left[(\mathrm{Var}^{\rho(t)}(u))^2+\mathrm{Var}^{\rho(t)}(u)+|u^\ast_{\mathsf{A}}|\mathrm{Var}^{\rho(t)}(u)\right]\EE\left[|u^n_t|^2(\omega^n_t)^2\right]\\
&+\widetilde{C}\left[\left|u^\ast_{\mathsf{A}}-\overline{u}^{\rho(t)}\right|^2+1\right]\EE\left[|u^n_t|^2(\omega^n_t)^2\right]\\
&+\widetilde{C}\mathrm{Var}^{\rho(t)}(u)\left[\mathrm{Var}^{\rho(t)}(u)+\left|u^\ast_{\mathsf{A}}\right|+1\right]\EE\left[(\omega^n_t)^2\right]\,,
\end{aligned}
\end{equation}
where $\widetilde{C}$ is a constant depends on $\Lambda$, $\|\Gamma^{-1}\|_2$, $\|\Gamma^{-1}_0\|_2$, $|\mathsf{r}|$ and $M$.

Next we estimate $\EE(w^n_t)^2$. Note that, according to~\eqref{eqn:weki_SDE}, one has
\begin{equation}\label{omegainequality:new}
\frac{1}{2}\frac{\rd}{\rd t}\EE(\omega^n_t)^2 \leq \left(\|\mathcal{R}_1\|_{\infty}+\left|\mathcal{R}_2(u^n_t,t)\right|+\|\mathcal{R}_3\|_{\infty}\right)\EE(\omega^n_t)^2\,.
\end{equation}
To control these terms we apply~\eqref{boundr1:new},~\eqref{boundr3:new} and~\eqref{R2bestestimation}, which leads to:
\begin{equation}\label{changeofw2}
\begin{aligned}
&\frac{\rd}{\rd t}\EE(\omega^n_t)^2\\
\leq &\,C\left[(\mathrm{Var}^{\rho(t)}(u))^2+\mathrm{Var}^{\rho(t)}(u)\right]\EE(\omega^n_t)^2\\
&+C\left\{\left|\overline{u}^{\rho(t)}-u^\ast_{\mathsf{A}}\right|\left[\left|\overline{u}^{\rho(t)}-\Cov^{\rho(t)}_{u,u}(\Cov_\mathsf{A}(t))^{-1}u_\mathsf{A}\right|+\|\Cov_{u,u}^{\rho(t)}\|_2\Lambda_1\right]\right\}\EE(\omega^n_t)^2\\
&+C\|\Cov^{\rho(t)}_{\Rm,u}\|_2(|u_\mathsf{A}|+\Lambda_1)\EE(\omega^n_t)^2\\
&+C\left(\left|\overline{u}^{\rho(t)}-u^\ast_{\mathsf{A}}\right|\left\|I-(\Cov_\mathsf{A})^{-1}\Cov^{\rho(t)}_{u,u}\right\|_2+\|\Cov^{\rho(t)}_{\Rm,u}\|_2\right)\EE\left[|u^n_t|(\omega^n_t)^2\right]\\
\leq &\,C\left[(\mathrm{Var}^{\rho(t)}(u))^2+\mathrm{Var}^{\rho(t)}(u)\right]\EE(\omega^n_t)^2\\
&+C\left\{\left|\overline{u}^{\rho(t)}-u^\ast_{\mathsf{A}}\right|\left[\left|\overline{u}^{\rho(t)}-\Cov^{\rho(t)}_{u,u}(\Cov_\mathsf{A}(t))^{-1}u_\mathsf{A}\right|+\|\Cov_{u,u}^{\rho(t)}\|_2\Lambda_1\right]\right\}\EE(\omega^n_t)^2\\
&+C\|\Cov^{\rho(t)}_{\Rm,u}\|_2(|u_\mathsf{A}|+\Lambda_1)\EE(\omega^n_t)^2\\
&+C\left(\left|\overline{u}^{\rho(t)}-u^\ast_{\mathsf{A}}\right|\left\|I-(\Cov_\mathsf{A})^{-1}\Cov^{\rho(t)}_{u,u}\right\|_2+\|\Cov^{\rho(t)}_{\Rm,u}\|_2\right)\EE(\omega^n_t)^2/2\\
&+C\left(\left|\overline{u}^{\rho(t)}-u^\ast_{\mathsf{A}}\right|\left\|I-(\Cov_\mathsf{A})^{-1}\Cov^{\rho(t)}_{u,u}\right\|_2+\|\Cov^{\rho(t)}_{\Rm,u}\|_2\right)\EE\left[|u^n_t|^2(\omega^n_t)^2\right]/2\,,
\end{aligned}
\end{equation}
where $C$ is a constant depends on $\Lambda$, $\|\Gamma^{-1}\|_2$, $\|\Gamma^{-1}_0\|_2$, $|\mathsf{r}|$ and $M$.

Combine~\eqref{changeofu2w2} and~\eqref{changeofw2}, we have
\[
\frac{\rd }{\rd t}\left(
\begin{aligned}
& \EE|u^n_t|^2(\omega^n_t)^2\\
&\EE(\omega^n_t)^2
\end{aligned}
\right) \leq W(t)\left(
\begin{aligned}
& \EE|u^n_t|^2(\omega^n_t)^2\\
&\EE (N\omega^n_t)^2
\end{aligned}
\right)\,.
\]
Multiply $N^2$ on both sides of this inequality and notice $\EE(N\omega^n_t)=1$, we have, for $0\leq t\leq 1$
\[
\frac{\rd }{\rd t}\left(
\begin{aligned}
& \EE|u^n_t|^2(N\omega^n_t)^2\\
&\EE(N\omega^n_t)^2-(\EE N\omega^n_t)^2+1
\end{aligned}
\right) \leq W(t)\left(
\begin{aligned}
& \EE|Nu^n_t|^2(\omega^n_t)^2\\
&\EE (N\omega^n_t)^2-(\EE N\omega^n_t)^2+1
\end{aligned}
\right)\,,
\]
which concludes~\eqref{betterboundforvarianceofweight}, hence the proposition.
\end{proof}

\bibliographystyle{AIMS}
\bibliography{enkf}
\medskip
Received xxxx 20xx; revised xxxx 20xx.
\medskip

\end{document}